\documentclass[11pt, reqno]{amsart}
\usepackage{amsthm}
\usepackage{amsmath,amssymb,amsthm,mathrsfs}
\usepackage{mathptmx}
\usepackage{MnSymbol}
\usepackage{color}
\usepackage[all]{xy}
\usepackage[english]{babel}
\usepackage{graphicx}
\usepackage{todonotes}
\usepackage[T1]{fontenc}
\usepackage[final=true]{hyperref}

\theoremstyle{plain}
\newtheorem{thm}{Theorem}[section]
\newtheorem{lem}[thm]{Lemma}
\newtheorem{cor}[thm]{Corollary}
\newtheorem{prop}[thm]{Proposition}

\newtheorem{assumption}[thm]{Assumption}
\theoremstyle{definition}
\newtheorem{ex}[thm]{Example}
\newtheorem{rem}[thm]{Remark}
\theoremstyle{definition}
\newtheorem{defi}[thm]{Definition}

\DeclareMathOperator{\Hom}{Hom}

\DeclareMathOperator{\Ker}{Ker}
\DeclareMathOperator{\modd}{mod}
\DeclareMathOperator{\HH}{H}
\DeclareMathOperator{\mi}{m_i}

\DeclareMathOperator{\Repp}{Rep}
\DeclareMathOperator{\In}{in}
\DeclareMathOperator{\out}{out}

\DeclareMathOperator{\Coker}{Coker}
\DeclareMathOperator{\im}{Im}
\DeclareMathOperator{\Ir}{Irr}
\DeclareMathOperator{\id}{id}
\DeclareMathOperator{\wt}{wt}
\DeclareMathOperator{\m}{m}
\DeclareMathOperator{\pl}{pl}
\DeclareMathOperator{\pr}{pr}
\DeclareMathOperator{\End}{End}
\DeclareMathOperator{\Ext}{Ext}
\DeclareMathOperator{\D}{D}

\newcommand{\kg}{\unlhd}
\newcommand{\Qu}{\mathrm{Q}}

\begin{document}

\title{Homological description of crystal structures on Lusztig's quiver varieties}
\author{Bea Schumann}
\address{Mathematical Institut, University of Cologne}
\email{bschuman@math.uni-koeln.de}

\begin{abstract}
Using methods of homological algebra, we obtain an explicit crystal isomorphism between two realizations of crystal bases of the lower part of the quantized enveloping algebra of (almost all) finite dimensional simply-laced Lie algebras. The first realization we consider is a geometric construction in terms of irreducible components of certain quiver varieties established by Kashiwara and Saito. The second is a realization in terms of isomorphism classes of quiver representations obtained by Reineke using Ringel's Hall algebra approach to quantum groups. We show that these two constructions are closely related by studying sufficiently generic representations of the preprojective algebra.
\end{abstract}

\thanks{This research was supported by the Graduiertenkolleg ''Global Structures in Geometry und Analysis'' and the DFG Priority program Darstellungstheorie 1388.}
\thanks{The author is deeply indebted to Volker Genz and Markus Reineke for several valuable and clarifying discussion and comments. Further a big thanks goes to Michael Ehrig for careful proofreading and helpful advise and Yoshiyuki Kimura for his comments on an earlier version of this paper.}

\maketitle

\section*{Introduction}
\subsection*{BACKGROUND}
Since Gabriel's famous theorem in 1972 (see \cite{G}) the connection between representation theory of quivers and Lie algebras of simply-laced type has evolved into a rich area of research. In loc. cit. Gabriel classified all quivers of finite representation type (i.e. with finitely many isomorphism classes of indecomposable representations). He showed that a quiver $\Qu$ is of finite representation type if and only if its underlying diagram is a simply-laced Dynkin diagram (i.e. of type $A_n$, $D_n$, $E_6$, $E_7$ or $E_8$). Furthermore, for Dynkin quivers, there is a bijection between the isomorphism classes of indecomposable representations and the set of negative roots of the Lie algebra associated to the underlying Dynkin diagram.

This theorem was used by Ringel (see e.g. \cite{hall}) to establish an even deeper connection. In loc. cit. Ringel showed that there is a $\mathbb{Q}(v)$-algebra isomorphism between the (twisted, generic) Hall algebra $\mathscr{H}(\Qu)$ and the negative part $U_v(\mathfrak{n}^-)$ of the quantized enveloping algebra of the simply-laced finite dimensional Lie algebra $\mathfrak{g}$ associated to $\Qu$. The underlying vector space of the Hall algebra has a basis consisting of all isomorphism classes of representations of $\Qu$. The multiplication in $\mathscr{H}(\Qu)$ is given in terms of certain numbers of filtrations of $\Qu$-representations over a finite field.

The identification of $U_v(\mathfrak{n}^-)$ with the Hall algebra yields a natural PBW-type basis of this algebra (w.r.t. a reduced decomposition of the longest Weyl group element adapted to $\Qu$) parameterized by the isomorphismen classes of $\Qu$-representations (see \cite{Ri2}).

Lusztig also constructed PBW-type bases of $U_v(\mathfrak{n}^-)$, one for each reduced decomposition of the longest Weyl group element $w_0$ of $\mathfrak{g}$, by applying certain braid group operators to Chevalley generators of $U_v(\mathfrak{n}^-)$ (see e.g. \cite{Lu93}). By this Lusztig constructed a unique basis of $U_v(\mathfrak{n}^-)$, called the canonical basis $\bf{B}$, with remarkable properties. Namely, let $\mathcal{B}$ be a fixed but arbitrary PBW-type basis, then the $\mathbb{Z}[v^{-1}]$-lattice $\mathscr{L}$ spanned by $\mathcal{B}$ is independent of the chosen reduced decomposition of $w_0$. Let $\pi: \mathscr{L}\rightarrow \mathscr{L}/v^{-1}\mathscr{L}$ denote the projection and let $\bar{\ }:U_v(\mathfrak{n}^-)\rightarrow U_v(\mathfrak{n}^-)$ be the canonical $\mathbb{Q}$-algebra involution of $U_v(\mathfrak{n}^{-})$ sending the generator $F_i$ to $F_i$ and $v$ to $v^{-1}$. Then the canonical basis is the unique $\bar{\ }$ invariant basis whose image under $\pi$ coincides with the image of $\mathcal{B}$.

For ${\bf i}=(i_1,i_2,\ldots,i_N)$ we denote by $\mathcal{B}_{\bf{i}}$ the PBW basis corresponding to the reduced decomposition $w_0=s_{i_1}s_{i_2}\cdots s_{i_N}$. Lusztig defined a combinatorial colored graph (or crystal) structure on $\bf{B}$, called $B(\infty)$, where the vertices are given by the basis elements of a fixed PBW-type basis $\mathcal{B}_{\bf{i}}$ and thus indexed by $\mathbb{N}^N$ (note that $N$ is the number of negative roots of $\mathfrak{g}$). We call the image of $b\in \mathcal{B}_{\bf{i}}$ under the map
\begin{align*}
\Psi_{\bf{i}}:\ \ \ \ \ \ \ \ \ \ \ \ \ \ \ \ \  \mathcal{B}_{\bf{i}} & \rightarrow \mathbb{N}^N \\
 F_{\beta_{1}}^{n_1}F_{\beta_{2}}^{n_2}\cdots F_{\beta_{N}}^{n_N} & \mapsto (n_1,n_2,\ldots, n_N)
\end{align*}
an $\bf{i}$-Lusztig parametrization of $b$. The arrows of the graph $B(\infty)$ are given by certain operators $\tilde{f}_i$, one for each simple root $\alpha_i$ of $\mathfrak{g}$, which form a combinatorial counterpart of the Chevalley generators of $U_v(\mathfrak{n}^-)$. The action of $\tilde{f}_{i_1}$ on $b\in \mathcal{B}_{\bf{i}}$ is easily defined if $\mathcal{B}_{\bf{i}}$ corresponds to a reduced expression of $w_0$ which starts with the simple reflection $s_{i_1}$. To determine the actions of $\tilde{f}_{i_k}$ for $k\ne i$ one uses the fact that two PBW-bases are obtained from each other by applying a composition of certain braid moves. The induced map $\gamma_{\bf{i},\bf{j}}$ such that $\gamma_{\bf{i},\bf{j}}\circ\Psi_{\bf{i}}=\Psi_{\bf{j}}$ corresponds to a composition of piecewise linear bijections of $\mathbb{N}^N$. For one braid move the transition map is explicitly known (see \cite{Lu93}). However the combinatorics behind the composition of these piecewise linear bijections is difficult to handle since it is given by the composition of operations that involve taking minima.

In \cite{Rei} Reineke used Ringel's Hall algebra approach to quantum groups to describe the crystal structure on a fixed PBW-basis $\mathcal{B}_{\bf{i}}$ explicitly in terms of the set of isomorphism classes of representations of a fixed Dynkin quiver $\Qu$. Here $\bf{i}$ corresponds to a reduced expression of $w_0$ which is adapted to a Dynkin quiver $\Qu$ that is satisfying a certain homological condition (see Definition \ref{special}).

Another construction of $B(\infty)$, called $B^g(\infty)$, in terms of quivers was given geometrically by Kashiwara and Saito in \cite{KS}. They defined a crystal structure on the irreducible components of the varieties $\Lambda_V$ of representations of the preprojective algebra $\Pi(\Qu)$ with varying underlying vector spaces $V$. These varieties were used by Lusztig to give a geometric construction of $U(\mathfrak{n}^-)$, the enveloping algebra of the Lie algebra $\mathfrak{n}^-$ (see \cite{Lu2}), as well as to show the existence of a nicely behaved basis of this algebra, called the semicanonical basis (see \cite{Lu00}), which is naturally indexed by the irreducible components of $\Lambda_V$. In the following we refer to these varieties as Lusztig's quiver varieties. Lusztig conjectured that the semicanonical and the specialization of the canonical basis coincides but Kashiwara and Saito gave a counterexample in \cite{KS}.

The Kashiwara operators corresponding to the arrows in the graph $B^g(\infty)$ can be described in the following way. Let $x\in X$ be a generic point of an irreducible component $X$ of Lusztig's quiver variety and let $M(x)$ be the corresponding $\Pi(\Qu)$-module. The component $\tilde{f}_iX$ is given as the closure of all points $y$ that (regarded as $\Pi(\Qu)$-modules $M(y)$) appear as the middle term of exact sequences
$$0\rightarrow M(x)\rightarrow M(y)\rightarrow S(i) \rightarrow 0,$$
where $S(i)$ is the simple representation of $\Pi(\Qu)$ corresponding to the vertex $i$.

\subsection*{MAIN RESULT}

The main result of this paper is a homological interpretation of the geometric realization of $B(\infty)$ of \cite{KS} using the representation theory of the preprojective algebra and its underlying Dynkin quiver. We compute the geometric crystal structure using the combinatorics of the Auslander-Reiten quiver of the Dynkin quiver associated to the Lie algebra $\mathfrak{g}$. To achieve this we use results by Reineke and give an explicit crystal isomorphism between the geometric and the homological realization of the crystal $B(\infty)$ given in \cite{Rei}. This provides an explanation for the compatibility of the indexation of the canonical and the semicanonical basis.

We remark that this result can also be deduced by the work of Baumann-Kamnitzer (see \cite{BK}) and Saito (see \cite{Sai}) that both provide a crystal isomorphism between the realization of $B(\infty)$ in terms of $\bf{i}$-Lusztig parametrization and the geometric realization. In \cite{BK} Baumann and Kamnitzer used some reflection functors for representations of the preprojective algebra and the classical Bernstein-Gelfand-Ponomarev reflection functors for representation of Dynkin quivers. In \cite{Sai} Saito used a result of Kimura to show that changing the orientations of the underlying Dynkin quiver of the preprojective algebra such that $i_1$ is a sink is compatible with the transition to an $\bf{i}$-Lusztig parametrization such that the reduced decomposition of $w_0$ starts with $i_1$.

The methods in this work are quite different from the ones used in \cite{BK} and \cite{Sai}. Here we take a direct approach fixing the orientation of the Dynkin quiver $\Qu$. We develop a combinatorial method to determine the dimension of the head of a generic module of an irreducible component $X$ of $\Lambda_V$ and construct a dense subset of $X$ for which the action of the crystal operator $\tilde{f}_i$ can be determined.

\subsection*{STRUCTURE OF THIS PAPER}
The paper is structured as follows.

We recall the geometric construction of $B^g(\infty)$ via the irreducible components of Lusztig's quivers varieties in Section \ref{geometricbinfty}.

In Section \ref{homologicbinfty} Reineke's construction of $B(\infty)$  is introduced. We denote this crystal by $B^{\mathscr{H}}(\infty)$. Each vertex $b\in B^{\mathscr{H}}(\infty)$ is given by an isomorphism class $[M]$ for $M$ a $\Qu$-representation. Reineke showed that if there is an exact sequence of $\Qu$-representations
$$0\rightarrow M \rightarrow X \rightarrow S(i)\rightarrow 0$$
such that $X$ satisfies a certain additional property, we have $\tilde{f}_i{[M]}=[X]$.
He further proved that for any $\Qu$-representation $M$ there exists a representation $X$ satisfying these properties. The existence is deduced from the classification of the middle terms of short exact sequences of $\Qu$-representations ending in the simple $\Qu$-representation $S(i)$. Thereby an algorithm is obtained to determine the actions of the Kashiwara operators in terms of the given combinatorial data $(\mu_B(M))_B$, where $B$ varies over all indecomposable $\Qu$-representations. Here $\mu_B(M)$ denotes the multiplicity of the indecomposable direct summands $B$ of $M$.

In Section \ref{comparebinfty} we give a crystal isomorphism between $B^{\mathscr{H}}(\Qu)$ and $B^g(\infty)$. We first describe how to relate the vertices. For this we use a result of Lusztig that gives a one-to-one correspondence between the irreducible components of quiver varieties and isomorphism classes of representations of $\Qu$ (see Proposition \ref{conormal}). We then work in a homological algebra setting using Ringel's description of $\Pi(\Qu)$-modules as pairs $(M,\phi)$ for $\phi\in\Hom(\tau^{-1}M,M)$, where $\tau^{-1}$ is the inverse Auslander-Reiten translation of $\Qu$. Let $M$ be a representation of $\Qu$ and $X_{[M]}$ be the irreducible component corresponding to the isomorphism class of $M$. We develop a combinatorial machinery to prove that the function $\varepsilon_i$ on $X_{[M]}$ (which counts how many consecutive times we applied $\tilde{f}_i$ to get to the desired vertex in the crystal graph) in the geometric setting only depends on the data $(\mu_B(M))_B$.
We further show that there is a dense subset of $X_{[M]}$ that is mapped to the component $X_{\tilde{f}_i[M]}$ by $\tilde{f}_i$. This is proved by constructing a certain class of points of an irreducible component which are sufficiently generic but can be handled combinatorially. We therefore obtain that the one-to-one correspondence between the irreducible components of quiver varieties and isomorphism classes of representations of $\Qu$ is indeed a crystal isomorphism (see Theorem \ref{mainresult}).
\bigskip

\paragraph{NOTATION}

In what follows $\Qu$ always denotes a Dynkin quiver with path algebra $k\Qu$ over a field $k$. For $M, N \in k\Qu-\modd$, we denote the isomorphism class of $M$ by $[M]$ and multiplicity of $N$ as a direct summand of $M$ by $\mu_N(M)$. We denote by $\tau M$ (resp. $\tau^{-1}M$) the Auslander-Reiten translation (resp. inverse Auslander-Reiten translation) of $M$. Further $\left<M,N\right>_R := \dim\Hom_{k\Qu}(M,N)-\dim\Ext^{1}(M,N)$ denotes the Ringel form and for $i$ in the vertex set of $\Qu$ we denote by $S(i)$ the simple module obtained by assigning a one-dimensional vector space to the vertex $i$ and a zero-dimensional vector space to all other vertices of $\Qu$.
\section{THE GEOMETRIC CONSTRUCTION}\label{geometricbinfty}

In this section we review the realization of the crystal graph $B(\infty)$ via Lusztig's quiver varieties. We start by introducing the varieties of interest.

\subsection{LUSZTIG'S QUIVER VARIETY}

For a Dynkin quiver $\Qu=(I, \Qu_1)$ we denote the associated \emph{double quiver} by $\overline{\Qu}=(I,H)$ that has as the set of vertices the vertices of $\Qu$ and for each arrow of $\Qu$, $H$ contains two arrows with the same endpoints, one in each direction.

For an arrow $h\in H$, we denote by $\overline{h}$ the arrow with $\out(h)=\In(\overline{h})$ and $\In(h)=\out(\overline{h})$. Therefore, the double quiver $\overline{\Qu}=(I,H)$ has as set of arrows $H=\Qu_1 \sqcup \overline{\Qu_1}$.

\begin{ex} We give an example of the double quiver of Dynkin type $A_3$:
\begin{equation*}
\overline{\Qu}=\xymatrix{
1\ar@<0.5ex>[r]^{\overline{h}_1} & 2\ar@<0.5ex>[l]^{h_1} \ar@<0.5ex>[r]^{\overline{h}_2} & 3.\ar@<0.5ex>[l]^{h_2}
}\
\end{equation*}
\end{ex}
 We define the function
\begin{align*}
\epsilon: H & \rightarrow \{\pm 1\} \\
h&\mapsto \left\{\begin{array}{cl} 1, & \mbox{if }h\in \Qu\\ -1, & \mbox{if }h \notin \Qu .\end{array}\right.
\end{align*}
The \emph{preprojective algebra $\Pi(\Qu)$} of a Dynkin quiver $\Qu$ is the quotient of the path algebra of the double quiver $\overline{\Qu}$ by the ideal generated by
$$\sum_{h\in H}\epsilon(h)h\bar{h}.$$
For a fixed finite--dimensional $I$--graded vector spaces $V=\bigoplus_{i\in I} V_i$ over $\mathbb{C}$, we define \emph{Lusztig's quiver variety} $\Lambda_V$ to be the variety of representations of $\Pi(\Qu)$ with underlying vector space $V$, i.e.
$$\Lambda_V:=\left\{x=(x_h)_{h\in H}\in \displaystyle \bigoplus_{h\in H}{\Hom(V_{\out(h)},V_{\In(h)})}\mid \sum_{h\in H, \In(h)=i}\epsilon(h)x_hx_{\bar{h}}=0 \text{ for all }i \in I \right\}.$$
Let $\Repp_V(\Qu)$ be the variety of representations of the Dynkin quiver $\Qu$ with underlying vector space $V$, that is
$$\Repp_V(\Qu)=\displaystyle \bigoplus_{h\in \Omega}{\Hom(V_{\out(h)},V_{\In(h)})},$$
which is clearly a closed subvariety of the affine variety $\Lambda_V$.

From now on we constantly identify the points of $\Lambda_V$ (resp. $\Repp_{V}$) with the corresponding modules over $\Pi(\Qu)$ (resp. $\mathbb{C}\Qu$) and write expressions like $M\in \Lambda_V$ for $M=(V,x)\in \Pi(\Qu)-\modd$.

We have an action of the group $G_v=\prod_i{GL(V_i)}$ on $\Lambda_V$ and $\Repp_{V}(\Qu)$ by base change, that is for $M=(V,x) \in \Lambda_V$, $g.M=\widetilde{M}$, where $\widetilde{M}=(V,\widetilde{x})\in \Pi(\Qu)-\modd$ with
$$\widetilde{x}_h:=g_{\In(h)}x_hg^{-1}_{\out(h)}$$
and analogously for $M\in \Repp_{V}(\Qu)$.

The orbits of this action on $\Lambda_V$ (resp. $\Repp_{V}(\Qu)$) are exactly the isomorphism classes of representations of $\Pi(\Qu)-\modd$ (resp. $k\Qu-\modd$) with a fixed dimension vector $v$.

\begin{rem} The definition of Lusztig's quiver variety given in \cite{KS} imposes an additional nilpotency condition on the elements of $\Lambda_V$. But, since we restrict ourselves to preprojective algebras of Dynkin quivers, this condition is automatically satisfied (see \cite[Proposition 14.2(a)]{Lu}) and we thus omit it.
\end{rem}

Note that, up to isomorphism, $\Lambda_V$ depends only on the graded dimension $v=(\dim V_i)_{i\in I}$ of $V$. Therefore we also denote $\Lambda_V$ by $\Lambda(v)$,  regarding the graded dimension of the vector spaces as part of the datum of the representations of $\Pi(\Qu)$.

\subsection{KASHIWARA OPERATORS}

Following \cite{KS}, we recall the crystal structure on the set of irreducible components of $\Lambda(v)$, which we denote by $\Ir\Lambda(v)$.

For $i\in I$ and $M\in\Lambda(v)$, define $\varepsilon_i(M)$ to be the dimension of the $S(i)$-isotypic component of the head of $M$. For $M=(V,x)\in\Pi(\Qu)-\modd$, that is
\begin{equation}\label{eq:epsil}
{\varepsilon_{i}(M)}=\dim\Coker\left(\displaystyle\bigoplus_{h:\In(h)=i}V_{\out(h)}\xrightarrow{x_h}{V_i}\right).
\end{equation}
For $c\in \mathbb{Z}_{\ge 0}$, we further introduce the subsets
\begin{equation*}
\Lambda(v)_{i,c}:=\{M\in \Lambda(v)  \mid \varepsilon_i(M)=c\}.
\end{equation*}
To describe the actions of the Kashiwara operators, let $e^{i}\in \mathbb{Z}_{\ge 0}^{I}$ be such that $e^{i}_j=\delta_{ij}$ and fix $c\in \mathbb{Z}_{\ge 0}$ and $v\in\mathbb{Z}_{\ge 0}^I$ such that $v_i-c\ge 0$. We define $$\tilde{\Lambda}(v,c,i):=\{M,N,\varphi \mid M\in \Lambda(v)_{i,c}, N\in \Lambda(v-ce^i)_{i,0}, \varphi\in \Hom_{\Pi(\Qu)}(N,M) \text{ injective}\}.$$
Considering the diagram

\begin{equation}\label{diaggeom}
\Lambda(v-ce^i)_{i,0} \xleftarrow{p_1} \tilde{\Lambda}(v,c,i)\xrightarrow{p_2}\Lambda(v)_{i,c},
\end{equation}
where $p_1(M,N,\varphi)=N$ and $p_2(M,N,\varphi)=M$, it is shown in \cite[Lemma 5.2.3]{KS} that the map $p_2$ is a principal $G_v$-bundle and the map $p_1$ is smooth with a connected variety as fiber. Standard algebraic geometry arguments then show (for $\Lambda(v)_{i,c}\ne\emptyset$) that there is a one--to--one correspondence between the set of irreducible components of $\Lambda(v-ce^i)_{i,0}$ and the set of irreducible components of $\Lambda(v)_{i,c}$.

For $i\in I$  the function $\varepsilon_i$ given in (\ref{eq:epsil}) is upper semicontinuous, thus for each $X\in \Ir\Lambda(v)$ there is an open dense subset of $X$ such that $\varepsilon_i$ is constant (namely the value of $\varepsilon_i$ of this subset is the minimal value of $\varepsilon_i$ on $X$). For $X\in\Ir\Lambda(v)$ we thus define
\begin{equation}\label{eq:epsilon}\varepsilon_i(X):=\min_{M\in X}\varepsilon_i(M).
\end{equation} We also define for $c\in \mathbb{Z}_{\ge 0}$
$${ \Ir\Lambda(v)_{i,c}}:=\{X\in \Ir\Lambda(v)\mid\varepsilon_i(X)=c\}.$$
We get Bijection (\ref{eq:irrbijectionn}) directly from the prior considerations and \cite[Theorem 12.3 b)]{Lu} which states that  both $\Lambda_v$ and $\Lambda(v)_{i,c}$ have pure dimension $\frac{1}{2}\dim E_v$:
\begin{equation}\label{eq:irrbijectionn}
\Ir\Lambda(v-ce^i)_{i,0}\cong \Ir\Lambda(v)_{i,c}.
\end{equation}
\begin{rem}\label{dense} Note that the image of an irreducible component $X\in \Ir\Lambda(v)_{i,c}$ is the unique irreducible component $Y\in\Ir\Lambda(v-ce^i)_{i,0}$ such that for any dense subset $\mathscr{D}$ of $X$, we have $p_2^{-1}p_1(\mathscr{D})\subset Y$.
\end{rem}
Suppose that $\bar{X}\in \Ir\Lambda(v-ce^i)_{i,0}$ corresponds to $X\in \Ir\Lambda(v)_{i,c}$ by Bijection (\ref{eq:irrbijectionn}). We define maps
\begin{align*}
\tilde{f}_i^c:\Ir\Lambda(v-ce^i)_{i,0}&\rightarrow \Ir\Lambda(v)_{i,c},\\
\tilde{e}_i^{\max}:\Ir\Lambda(v)_{i,c}&\rightarrow \Ir\Lambda(v-ce^i)_{i,0} \label{actione}
\end{align*}
by
$$\tilde{f}_i^c(\bar{X}):=X \text{ and }\tilde{e}_i^{\max}(X):=\bar{X}.$$
The data of these maps yields a crystal structure on $B^g(\infty):=\bigsqcup_{v}\Ir\Lambda(v)$ together with defining for $X\in B^g(\infty)$
\begin{align*}
\wt(X)&:=-\displaystyle\sum_{i\in I}v_i\alpha_i \text{ for } X\in \Ir\Lambda(v),\\
\varphi_i(X)&:=\varepsilon_i(X)+\left<h_i,\wt(X)\right>.
\end{align*}
It is shown in \cite[Theorem 5.3.2]{KS} that $B^g(\infty)$ is isomorphic to the crystal $B(\infty)$ of $U_v(\mathfrak{n}^-)$.

\section{THE HOMOLOGICAL CONSTRUCTION}\label{homologicbinfty}
\subsection{RINGEL HALL ALGEBRAS AND CRYSTAL BASES}\label{ringelhallalgebra}

We review the notion and some facts about Ringel Hall algebras. Even though we do not use the results of this section in our proofs, they form an indispensable ingredient for the homological realization of the crystal $B(\infty)$ introduced by Reineke in \cite{Rei}.

Let $\Qu$ be a Dynkin quiver of type $A,D,E$ and $\mathfrak{g}$ the associated finite dimensional complex simple Lie algebra with Cartan decomposition $\mathfrak{g}=\mathfrak{n}^{+}\oplus \mathfrak{h} \oplus \mathfrak{n}^{-}$ and let $k$ be an arbitrary field. Recall that by Gabriel's Theorem (see \cite{G}) we have a one--to--one correspondence between the set of negative roots of $\mathfrak{g}$ and the isomorphism classes of indecomposable $k\Qu$-modules, independent of the ground field $k$. For a $k\Qu$-module $M$ we denote the isomorphism class by $[M]$.

Recall that for an indecomposable $k\Qu$-module $B$ and any $k\Qu$-module $M$, we denote by $\mu_B(M)$ the multiplicity of $B$ as a direct summand of $M$. For $\alpha$ a negative root of $\mathfrak{g}$ and any field $k$, we denote a fixed representative of the isomorphism class of indecomposable $k\Qu$-modules associated to this root by $M(\alpha,k)$. Let $R^{-}$ be the negative root lattice of $\mathfrak{g}$. This yields a one-to-one correspondence between isomorphism classes of $k\Qu$-modules and maps with finite support $\gamma: R^{-} \rightarrow \mathbb{Z}_{\ge 0}$ by mapping $[M]$ ($M\in k\Qu-\modd$) to $\gamma_{M}:\alpha \mapsto \mu_{M(\alpha,k)}(M)$.  Conversely, for a map $\gamma:R^{-}\rightarrow \mathbb{Z}_{\ge 0}$, we get a representative $M$ of an isomorphism class of $k\Qu$-modules via $M=\bigoplus_{\alpha \in R^{-}}M(\alpha,k)^{\gamma(\alpha)}$ which we denote by $M(\gamma,k)$.

Let $q$ be a prime power, $M$, $N$, $X$ be $k\Qu$--modules and $\mathbb{F}_q$ be the finite field with $q$ elements. We define ${F_{M,N}^{X}(q)}$ as the number of submodules $U$ of $M(\gamma_X,\mathbb{F}_q)$ over $\mathbb{F}_q$ such that $U\cong M(\gamma_N,\mathbb{F}_q)$ and $M(\gamma_X,\mathbb{F}_q)/U \cong M(\gamma_M,\mathbb{F}_q)$ over $\mathbb{F}_q$.

We have that ${F_{M,N}^{X}(q)}$ is a polynomial in $q$ (seen as a formal variable) with integer coefficients called the \emph{Hall polynomial} (\cite[Theorem 1, p. 439]{Ri90}).

Setting $q=v^2$, we obtain:
$${F_{M,N}^{X}(v^2)}\in\mathbb{Q}[v].$$
Recall that for $M,N\in k\Qu-\modd$ we have $\left<M,N\right>_R=\dim\Hom(M,N)-\dim\Ext^{1}(M,N)$. We define the (twisted, generic) \emph{Hall algebra} $\mathscr{H}(\Qu)$ of a quiver $\Qu$ to be the $\mathbb{Q}[v,v^{-1}]$-vector space with basis elements $u_{[M]}$ indexed by the isomorphism classes $[M]$ of $k\Qu$--modules and multiplication defined by
$$u_{[M]}\star u_{[N]}:=v^{\left<M,N\right>_R}\sum_{[X]}F^{X}_{M,N}(v^2)u_{[X]}.$$
For $m\in \mathbb{Z}_{\ge 0}$ we introduce the following abbreviations:
\[ \begin{matrix} [m]!:=\displaystyle\prod_{k=1}^m[m] & & \text{ and } & & [m]:=\frac{v^m - v^{-m}}{v-v^{-1}}. \end{matrix} \]
By \cite[Proposition, p. 21]{hall} the $\mathbb{Z}[v,v^{-1}]$-algebra $\mathscr{H}(\Qu)$ is generated by the elements $u_{[S(i)]}^{(\star m)}$ with $1\le i \le n$ and $m\ge 1$. Here (see loc cit., p.16)
$$u_{[S(i)]}^{(\star m)}:=\frac{1}{[m]!}u_{[S(i)]}^{\star m}.$$
Let $A=(a_{i,j})_{i,j\in I}$ be the Cartan matrix of $\mathfrak{g}$ and let $I=\{1,2,\ldots,n\}$. We define $U'_{q}(\mathfrak{n}^{-})$ as the $\mathbb{Q}(v)$-algebra with generator $F_1,F_2,\ldots ,F_n$  and relations
\begin{align*}
F_i F_j - F_j F_i & = 0 &\text{     if }&a_{i,j}=0, \\
F_i^2 F_j - (v+v^{-1})F_i F_j F_i + F_j F_i^2 & = 0 & \text{     if } &a_{i,j}=-1
\end{align*}
We denote
$$F_i^{(m)}:=\frac{1}{[m]!}F_i^{m}.$$
The $\mathbb{Z}[v,v^{-1}]$-subalgebra of $U'_{q}(\mathfrak{n}^{-})$ which is generated by the elements $F_i^{(m)}$ ($1\le i \le n$ and $m\ge 0$) is denoted by $U_q(\mathfrak{n}^-)$.

We have the following theorem by Ringel (see e.g. \cite[Theorem, p. 21]{hall}):

\begin{thm}\label{hall} The map $\eta_{\Qu}: \mathscr{H}(\Qu) \rightarrow U_v(\mathfrak{n}^{-})$ defined by $\eta_{\Qu}(u_{[S(i)]})=F_i$ induces an isomorphism of $\mathbb{Z}[v,v^{-1}]$-algebras.
\end{thm}

By setting $f^{\Qu}_{[M]}=v^{\dim\End(M)-\dim M}u_{[M]}$ we get a basis of $\mathscr{H}(\Qu)$. Via $\eta_{\Qu}$, this basis is sent to a PBW-basis $B_{\Qu}$ of $U_v(\mathfrak{n}^{-})$ corresponding to a reduced expression of the longest Weyl group element $w_0$ of $\mathfrak{g}$ adapted to $\Qu$ (\cite[Theorem 7]{Ri2}). We denote those basis elements by $F^{\Qu}_{[M]}:=\eta_{\Qu}(f^{\Qu}_{[M]})$.

The $\mathbb{Z}[v^{-1}]$-lattice $\mathscr{L}$ spanned by $B_{\Qu}$ is independent of the reduced decomposition of $w_0$ (\cite[Proposition 2.3]{Lu1}) and thus of the orientation of $\Qu$. Furthermore, by loc. cit., the image of $B_{\Qu}$ under the projection $\pi:\mathscr{L}\rightarrow \mathscr{L}$/$v^{-1}\mathscr{L}$ is a $\mathbb{Z}$-basis $B$ of $\mathscr{L}$/$v^{-1}\mathscr{L}$ which is again independent of the orientation of $\Qu$.

Let us denote by $\bar{\ }:U_v(\mathfrak{n}^-)\rightarrow U_v(\mathfrak{n}^-)$ the canonical $\mathbb{Q}$-algebra involution of $U_v(\mathfrak{n}^{-})$ sending the generator $F_i$ to $F_i$ and $v$ to $v^{-1}$. Then there is a unique $\bar{\ }$-invariant basis $\mathscr{B}$ of $\mathscr{L}$ whose image under $\pi$ is $B$ (\cite[Theorem 3.2]{Lu1}).

This basis is called the \emph{canonical basis}. The elements of $\mathscr{B}$ can hence be parametrized by the isomorphism classes of $k\Qu$-modules defining $\mathscr{F}^{\Qu}_{[M]}\in \mathscr{B}$ by
$$\pi(\mathscr{F}^{\Qu}_{[M]})=\pi(F^{\Qu}_{[M]})$$
and setting
$$\mathscr{B}=\{\mathscr{F}^{\Qu}_{[M]} \mid  M\in k\Qu-\modd\}.$$
We get the following from \cite[Theorem 2.3]{Lu2}.
\begin{thm} Let $\mathscr{B}$ be the canonical basis. Set
\begin{equation*}
\begin{matrix}
\mathcal{L}'=\displaystyle\bigoplus_{b\in \mathscr{B}}\mathcal{A}_0 b, & & & & B'=\{b \modd v^{-1} \mathcal{L}'; \ b\in\mathscr{B}\}.
\end{matrix}
\end{equation*}
Then $(\mathcal{L}',B')$ is a crystal basis of $U_v(\mathfrak{n}^-)$ and hence isomorphic to $(\mathcal{L}(\infty),\mathcal{B}(\infty))$.
\end{thm}
\begin{rem}\label{verticesisoclasses}
The considerations above yield a parametrization of the vertices of the crystal graph of $U_v(\mathfrak{n}^-)$ as isomorphism classes of $k\Qu$-representations. This identification is of crucial use in the homological construction of the crystal graph recalled in the following.
\end{rem}

\subsection{Kashiwara operators}

In \cite{Rei} the crystal graph of $U(\mathfrak{n}^-)$ is realized as the set of isomorphism classes of $k\Qu$-modules.
To state the main result of loc. cit., we need the following definitions where we adopt the notations of Subsection \ref{ringelhallalgebra}.

\begin{defi}\label{definitionai}
\begin{itemize}
\item
The \emph{degree} of a non--zero Laurent polynomial $c\in \mathbb{Z}$[$v,v^{-1}$] is the smallest $d\in \mathbb{Z}$ such that $v^{-d}c\in \mathbb{Z}$[$v^{-1}$].
\item
For a $k\Qu$-module $M$ and $i\in I$ we define
$${ a_i^{\Qu}(M)}:=\displaystyle\max_{{[X]}}\deg c_i^{\Qu}(M,X),$$
where $u_{[S(i)]}\cdot f^{\Qu}_{{[M]}}=\displaystyle\sum_{{[X]}}c_i^{\Qu}(M,X)f^{\Qu}_{[X]}$.
\item For $u\in U_v(\mathfrak{n}^{-})$, let $\rho(u)$ be the largest integer $r$ such that $u\in F_i^{r}U_v(\mathfrak{n}^{-}).$
\end{itemize}
\end{defi}

\begin{thm}[{\cite[Proposition 3.2]{Rei}}]\label{reinekethm} Let $M$ be a $k\Qu$--module and $i\in I$. If $X$ is a $k\Qu$--module such that
\begin{equation}\label{eq:reinekemain}
 c_i^{\Qu}(M,X)=a_i^{\Qu}(M)\ge a_i^{\Qu}(X)-1,
\end{equation}
 then $a_i^{\Qu}(M)= a_i^{\Qu}(X)-1$, $\rho(\mathscr{F}^{\Qu}_{[M]})=a_i^{\Qu}(M)$ and $\mathscr{F}^{\Qu}_{[X]}=\tilde{f}_i \mathscr{F}^{\Qu}_{[M]}\modd v^{-1}\mathscr{L}.$
\end{thm}
Thus, if the criterion (\ref{eq:reinekemain}) on the degree of the polynomial $c_i^{\Qu}(M,X)$ is fulfilled, the Kashiwara operator $\tilde{f}_i$ ($i\in I$) maps the isomorphism class $[M]$ to the isomorphism class $[X]$.

In \cite{Rei} it is proved that, for certain choice of orientations for $\Qu$, we can find for any $M\in k\Qu-\modd$ and any $i\in I$ such a $k\Qu$-module $X$ fulfilling the criterion. This is done by classifying all middle terms of short exact sequences of $k\Qu$-modules of the form
$$0\rightarrow M \rightarrow X \rightarrow S(i) \rightarrow 0,$$
which allows one to express the function $a_i^{\Qu}$ in terms of multiplicities of certain indecomposable direct summands of $M$. Let us recall these results in more details.

\begin{defi}\label{special} A quiver $\Qu$ is called \emph{special} if $\dim\Hom_{k\Qu} (X, S(i) ) \le 1$ for all $i \in I$ and all indecomposable $k\Qu$--modules $X$.
\end{defi}
For the rest of this section we make the following assumption:
\begin{assumption}
{\bf $\Qu$ is a fixed special Dynkin quiver.}
\end{assumption}
In Section \ref{specialclassification}, we examine this condition further and classify all possible orientations for Dynkin diagrams which yield a special Dynkin quiver. In particular, this shows that we can find at least one orientation that is special for any simply-laced type Dynkin quiver except $E_8$.

Fix $i\in I$, we introduce two sets of $k\Qu$--modules which play an important role in the following. We first define
 $$\mathscr{P}_i(\Qu):=\{X\in k\Qu-\modd\mid X \text{ is indecomposable and }\dim\Hom_{k\Qu}(X,S(i))\ne 0\}.$$

On $\mathscr{P}_i(\Qu)$ we have a relation $\kg$ given by
\begin{equation}\label{eq:preorder}
X \kg Y \iff \Hom_{k\Qu}(X,Y)\ne 0.
\end{equation}

The following proposition shows that this is a partial order on $\mathscr{P}_i(\Qu)$.

\begin{prop}[{\cite[Proposition 4.3.]{Rei}}]\label{partialorder} Let $X,Y$ be in $\mathscr{P}_i(\Qu)$. If there is a path from $[X]$ to $[Y]$ in the Auslander-Reiten quiver $\Gamma_{\Qu}$ of $\Qu$, then there exists a map $f\in \Hom_{k\Qu}(X,Y)$ inducing an isomorphism $\Hom_{k\Qu}(Y,S(i))\xrightarrow{\sim} \Hom_{k\Qu}(X,S(i))$. In particular, $\mathscr{P}_i(\Qu)$ is a poset.
\end{prop}

Recall that an \emph{antichain} is a subset of a poset such that no two (distinct) elements are comparable. We define
$$\mathscr{S}_i(\Qu):=\{V=\displaystyle\bigoplus_{j=1}^{k}X_j\mid\{X_1,X_2,\ldots,X_k\}\text{ is an antichain in }\mathscr{P}_i(\Qu)\}.$$
On $\mathscr{S}_i(\Qu)$ we have a partial order induced by the ordering ${\kg}$ on $\mathscr{P}_i(\Qu)$. By abuse of notation we denote this ordering also by $\kg$. It is given by $(V,V'\in\mathscr{S}_i(\Qu))$:
\begin{center}
$V \unlhd V'$ if and only if $\dim\Hom_{k\Qu}(B, V') \ne 0$ for each indecomposable direct summand $B$ of $V$.
\end{center}
Note that we always have $\mathscr{P}_i(\Qu) \subset \mathscr{S}_i(\Qu)$ as the set of trivial antichains.

\begin{ex}\label{15} We give two examples of the sets  $\mathscr{P}_i(\Qu)$ and  $\mathscr{S}_i(\Qu)$. See \cite[Section 8]{Rei} for more examples. Following Proposition \ref{partialorder}, the set $\mathscr{P}_i(\Qu)$ can be interpreted as a full subgraph of the Auslander-Reiten quiver. Recall that, by Gabriel's Theorem, the indecomposable $k\Qu$-modules are uniquely determined by their dimension vector. We thus use the dimension vector to denote the indecomposable module having this dimension.
\begin{enumerate}

\item For $\Qu=1 \leftarrow 2 \leftarrow 3$, the poset $\mathscr{P}_3(Q)$ \ is the union of all framed isomorphism classes of modules:
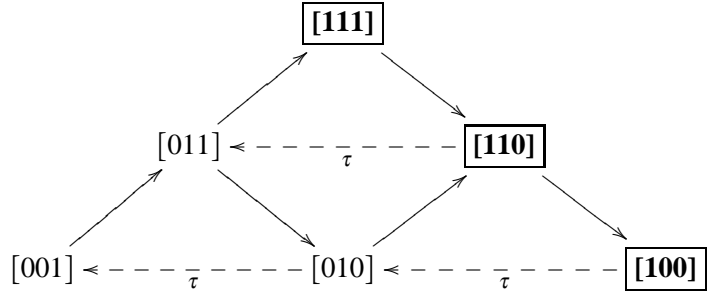
\begin{equation*}
\xymatrix{
& &\fbox{\bf{{[111]}}} \ar@{->}[rd]& &\\
&[011] \ar@{->}[ru] \ar@{->}[rd]& & \fbox{\bf{[110]}}\ar@{->}[rd]\ar@{-->}[ll]^{\tau} &\\
[001] \ar@{->}[ru]& & [010] \ar@{->}[ru] \ar@{-->}[ll]^{\tau} & & \fbox{\bf{[100]}}\ar@{-->}[ll]^{\tau} \\
}
\end{equation*}
Here $\mathscr{S}_3(\Qu)=\mathscr{P}_3(\Qu)$, i.e. $\mathscr{P}_3(\Qu)$ is a chain. The elements can be ordered as follows:
$$[111] \kg [110] \kg [100].$$

\item\label{Qalternatierend} Take $\Qu=1 \leftarrow 2 \rightarrow 3$. We find that $\mathscr{P}_2(\Qu)$ has the following shape:
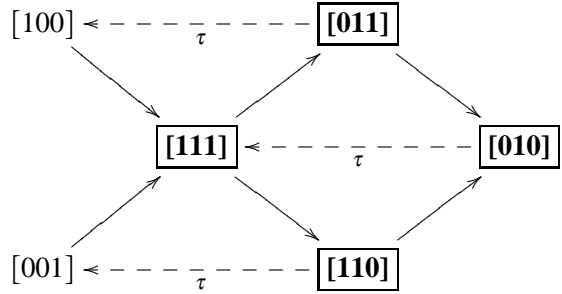
\begin{equation*}
\xymatrix{
[100]\ar@{->}[rd]& & \fbox{\bf{[011]}} \ar@{->}[rd]\ar@{-->}[ll]^{\tau} & \\
&\fbox{\bf{[111]}} \ar@{->}[ru] \ar@{->}[rd]& & \fbox{\bf{[010]}}\ar@{-->}[ll]^{\tau} \\
[001] \ar@{->}[ru]& & \fbox{\bf{[110]}} \ar@{->}[ru] \ar@{-->}[ll]^{\tau}  & \\
}
\end{equation*}
Here again we put a frame around every isormorphism class of elements of $\mathscr{P}_2(\Qu)$.

This time $\mathscr{P}_2(\Qu) \subsetneq \mathscr{S}_2(\Qu)$. We have a non--trivial antichain given by each $V\in [011 \oplus 110]$. So we have two maximal chains in $\mathscr{S}_2(\Qu):$
$$ [111] \kg [011] \kg [011 \oplus 110] \kg [010]$$
$$ [111] \kg [110] \kg [011 \oplus 110] \kg [010].$$
\end{enumerate}
\end{ex}

\begin{rem}\label{minimal} From the definition we directly get that there is always a unique $\kg$-minimal element in $\mathscr{S}_i(\Qu)$ for a fixed $i\in I$, namely the projective cover $P(i)$ of $S(i)$.
\end{rem}

We are now able to state the classification of middle terms of extension by $S(i)$. For that let $l(V)$ be the set of all $B\in \mathscr{P}_i(\Qu)$ which are minimal with the property that $B\ntrianglelefteq V$.

\begin{thm}[{\cite[Corallary 4.4, Proposition 4.5]{Rei}}]\label{U} Given a $k\Qu$-module $M$ and $i\in I$, the possible middle terms of exact sequences
$$0\rightarrow M \rightarrow X \rightarrow S(i)\rightarrow 0$$ are in $1:1$-correspondence with the elements $V\in \mathscr{S}_i(\Qu)$ such that $\tau B$ is a direct summand of $M$ for each $B\in l(V)$. The bijection is given via the map
$$V \mapsto X=N\oplus V,$$
where $M=N\oplus \bigoplus_{B\in l(V)} \tau B$.
\end{thm}

Recall that for $k\Qu$-modules $M$ and $B$ we denote by $\mu_B(M)$ the multiplicity of $B$ as a direct summand of $M$.
\begin{defi}\label{F(M,V)} Fix $i\in I$. For a $k\Qu$--module $M$ and an element $V\in\mathscr{S}_i(\Qu)$ define
\begin{equation*}
F_i(M,V):=\displaystyle\sum_{B\in\mathscr{P}_i(Q); \ B \kg V} \mu_B(M)-\mu_{\tau B}(M).
\end{equation*}
\end{defi}

Let $0\rightarrow M \rightarrow X \rightarrow S(i) \rightarrow 0$ be an exact sequence and $V\in \mathscr{S}_i(\Qu)$ corresponding to $X$ via the bijection given in Theorem \ref{U}. Then it is shown in \cite[Proposition 5.2]{Rei} that
\begin{equation}\label{eq:degree}
\deg{c_i}^{\Qu}(M,X)=F_i(M,V).
\end{equation}
In this language one verifies that the criterion used in Theorem \ref{reinekethm} is always fulfilled.

\begin{prop}[{\cite[Proposition 6.1]{Rei}}]\label{constructionV} Fix $i\in I$. Let $M$ be a $k\Qu$--module, $V_0\in \mathscr{S}_i(\Qu)$ such that $F_i(M,V_0)=a_i^{\Qu}(M)$ and $V_0$ is $\kg$--maximal with this property. Then $U_0:=\oplus_{B\in l(V_0)}\tau B$ is a direct summand of $M$. Set $X=M'\oplus V_0$ where $M=M'\oplus U_0$. Then
$$a_i^{\Qu}(X)= a_i^{\Qu}(M)+1.$$
\end{prop}

\begin{rem} Note that, by Theorem \ref{reinekethm}, this implies that there is a unique $\kg$-minimal antichain $V_0\in \mathscr{S}_i(\Qu)$ such that $F_i(M,V)$ is maximal.
\end{rem}

Recall from Remark \ref{verticesisoclasses} that the vertices of the crystal graph of $U_v(\mathfrak{n}^-)$ are parametrized by the isomorphism classes of $k\Qu$-representations. We therefore set
$$B^{\mathscr{H}}(\infty)=\{b_{[M]}\mid  M \in k\Qu-\modd\}.$$
The rest of this section is devoted to recalling how Theorem \ref{reinekethm} and Proposition \ref{constructionV} can be used to give a recipe for the actions of the Kashiwara operators on $B^{\mathscr{H}}(\infty)$. For further details see \cite[Chapter 7]{Rei}.

\begin{defi}\label{ee} Let $M$ be a $k\Qu$--module. We define ${\pl_i M }:=X$ where $X$ is obtained by the following recipe:
\begin{itemize}
\item For all $V\in \mathscr{S}_i(\Qu)$ compute the value
$$F_i(M,V)=\displaystyle\sum_{B\kg V} \mu_{B}(M)-\mu_{\tau B}(M).$$
\item Let $V_0$ be the $\kg$--maximal antichain where the maximal value of $F_i(M,V)$ is reached.
\item Let $U_0$ be the sum of all $\tau B$ such that $B\in \mathscr{P}_i(\Qu)$ and $B\ntrianglelefteq V_0$ minimally.
\item Set $X=M'\oplus V_0$ where $M=M'\oplus U_0.$
\end{itemize}
\end{defi}

\begin{rem} Note that $U_0$ must be a direct summand of $M$ by Proposition \ref{constructionV}.
\end{rem}

Thus we get by Theorem \ref{reinekethm} with (\ref{eq:degree}) for $b_{{[M]}}\in B^{\mathscr{H}}(\infty)$:
\begin{align}
\tilde{f}_ib_{[M]}&=b_{[\pl_iM]}, \label{eq:operatorf} \\
\varepsilon_i\left(b_{[M]}\right)&=a_i^{\Qu}(M). \label{eq:functione}
\end{align}
Using the description of the Kashiwara operator $\tilde{f}_i$ in (\ref{eq:operatorf}), we can determine the action of the partial inverse operator $\tilde{e}_i$ on $B^{\mathscr{H}}(\infty)$.

\begin{lem} Let $M$ be a $k\Qu$-module with the property that there exists an antichain $V\in \mathscr{S}_i(\Qu)$ such that $F_i(M,V)>0$. Let $V'_0$ be the $\kg$--minimal antichain with the property that $F_i(M,V'_0)=a_i^{\Qu}(M)$. Then $V'_0$ is a direct summand of $M$.
\end{lem}

\begin{proof}
Assume that $V'_0$ is not a direct summand of $M$.

First we deal with the case that for each indecomposable direct summand $B$ of $V'_0$ the equality $\mu_B(M)=0$ holds. We note that not all direct summands of $V_0'$ can be projective (otherwise we have a contradiction to $F_i(M,V'_0)>0$). Since the $\kg$-minimal elements of $\mathscr{P}_i(\Qu)$ is projective (see Remark \ref{minimal}), there exists $\widetilde{V}_0\in \mathscr{S}_i(\Qu)$ such that $\widetilde{V}_0\triangleleft V'_0$ and $F_i(M,V'_0)\le F_i(M,\widetilde{V}_0)$. A contradiction.

In case $\mu_B(M)\ne 0$ for a direct summand $B$ of $V'_0$, let $\widetilde{B}$ be a direct summand of $V'_0$ such that $\mu_{\widetilde{B}}(M)= 0$. Note that such a $\widetilde{B}$ exists by the assumption that $V'_0$ is not a direct summand of $M$. Let $V'_0=\widetilde{V_0} \bigoplus \widetilde{B}$. Then $F_i(M,V'_0)\le F_i(M,\widetilde{V_0})$ but $\widetilde{V_0}\kg V'_0$, once more a contradiction.
\end{proof}

Hence the following is well-defined.
\begin{defi}\label{operatorm} Let $M$ be a $k\Qu$-module with the property that there exists an antichain $V\in \mathscr{S}_i(\Qu)$ such that $F_i(M,V)>0$.  We define $ \mi M:=X'.$ where $X'$ is obtained by the following recipe:
\begin{itemize}
\item For all $V\in \mathscr{S}_i(\Qu)$  compute the value
$$F_i(M,V)=\displaystyle\sum_{B\kg V} \mu_{B}(M)-\mu_{\tau B}(M).$$
\item Let $V'_0$ be the $\kg$--minimal antichain where the maximal value of $F_i(M,V)$ is reached.
\item Let $U'_0$ be the sum of all $\tau B$ such that $B\in\mathscr{P}_i(\Qu)$ and $B\ntrianglelefteq V'_0$ minimally.
\item Set $X'=M''\oplus U'_0$ where $M=M'\oplus V'_0.$
\end{itemize}
\end{defi}

\begin{prop}\label{homologicalkashiwara} Let $M\in k\Qu-\modd$ have the property that there exists an antichain $V\in \mathscr{S}_i(\Qu)$ such that $F_i(M,V)>0$. Then $$\tilde{e}_ib_{{[M]}}=b_{{[\mi M]}}.$$
\end{prop}

\begin{proof} First we show that, since we know how the Kashiwara operator $\tilde{f}_i$ acts on $B^{\mathscr{H}}(\infty)$, the action of $\tilde{e}_i$ on $B^{\mathscr{H}}(\infty)$ is already determined by the equality
\begin{equation}\label{eq:actiondetermined}
\tilde{e}_i\tilde{f}_ib_{[M]}=b_{[M]}
\end{equation}
for all $M \in k\Qu-\modd$. For this, assume that Equation (\ref{eq:actiondetermined}) holds and let $N$ be a $k\Qu$-module such that there exists a $V\in \mathscr{S}_i(\Qu)$ with $F_i(N,V)>0$, i.e. $\varepsilon_i(b_{[N]})>0$ and $\tilde{e}_ib_{[N]}\in B^{\mathscr{H}}(\infty)$. Let $N'$ be a $k\Qu$-module such that
$$\tilde{f}_i\tilde{e}_ib_{[N]}=b_{[N']}.$$
 Applying $\tilde{e}_i$ yields
$$\tilde{e}_ib_{[N]}=\tilde{e}_ib_{[N']}.$$
Hence ${[N]}=[N']$.

Let $M \in k\Qu-\modd$ and let $\tilde{f}_ib_{[M]}=b_\text{[X]}$. Then $[X]=[\pl_i M]$, i.e. $X= M' \oplus V_0$ where $M = M' \oplus U_0$ and $U_0$, $V_0$ as in Definition \ref{ee}.

First we note that
$$F_i(X,V_0)=F_i(M,V_0)+F_i(V_0,V_0)-F_i(U_0,V_0).$$
It follows from the proof of \cite[Lemma 6.3]{Rei} and the considerations in loc. cit. p. 717 (since the graph $\Omega$ defined therein has no vertices in this case) that $F_i(V_0,V_0)-F_i(U_0,V_0)=1$, which yields
$$F_i(X,V_0)=F_i(M,V_0)+1=a_i^{Q}(M)+1.$$
Theorem \ref{reinekethm} together with (\ref{eq:degree}) then shows that the maximal value of $F_i(X,V)$ is reached at $V_0$. It remains to show, that $V_0$ is $\kg$-minimal with this property.

Let $V\in \mathscr{S}_i(\Qu)$ with $V\kg V_0$, then:
\begin{align*}
F_i(X,V_0)&=F_i(M,V_0)+1 \ge F_i(M,V)+1 \\
&= F_i(X,V)-F_i(V_0,V)+F_i(U_0,V) \\
& \ge F_i(X,V)-1,
\end{align*}
where the first inequality comes from the fact that the maximal value of $F_i(M,V)$ is reached at $V_0$ and the second inequality follows again from loc. cit. Lemma 6.3 and the considerations in loc. cit. p. 717.
\end{proof}

\subsection{SPECIAL QUIVERS}\label{specialclassification}

Let us examine the property special of a Dynkin quiver $\Qu$ more closely. In \cite[page 14]{Rei2}, a combinatorial description of special quivers is given. For that we need the following definition. A vertex $i\in I$ of $\Qu$ is called \emph{thick} if there exists an indecomposable $k\Qu$-representation $M=(V,x)$ such that $\dim V_i \ge 2$.

\begin{prop}[{\cite[Proposition 2.8]{Rei2}}]\label{combcond} Let $\Qu$ be a quiver. Then $\Qu$ is special if and only if no thick vertex is a source of $\Qu$.
\end{prop}

\begin{defi} Let $\mathfrak{g}$ be a Lie algebra of simply-laced type and $i\in I$. A fundamental weight $\omega_i$ of $\mathfrak{g}$ is called \emph{minuscule} if
$$-\left<\alpha,\omega_i\right>\le 1$$
for all negative roots $\alpha$.
\end{defi}
Recall that we denote by $M(\alpha,k)$ a representative of the isomorphism class of indecomposable $k\Qu$-modules that correspond to the negative root $\alpha$ by Gabriel's Theorem. We get from Proposition \ref{combcond}.
\begin{cor} Let $\Qu$ be a Dynkin quiver and $\mathfrak{g}$ the Lie algebra associated to the Dynkin diagram of $\Qu$. Then $\Qu$ is special if and only if for each vertex $i\in I$, that is a source of $\Qu$, the fundamental weight $\omega_i$ is minuscule.
\end{cor}
\begin{proof} Note that, if the vertex $i$ is a source of $\Qu$, we have for $\alpha$ a negative root and $\omega_i$ fundamental weight of $\mathfrak{g}$
$$-\left<\alpha,\omega_i\right>=\dim M(\alpha,k)_i \overset{i \text{ source}}{=}\dim\Hom_{k\Qu}(M(\alpha,k),S(i)),$$
where $\dim M(\alpha,k)_i$ denotes the dimension of the vector space assigned to vertex $i$ in the $\Qu$-representation $M(\alpha,k)$. Thus no thick vertex is a source of $\Qu$ if and only if letting $i$ run over all sources, we have $-\left<\alpha,\omega_i\right>\le 1$ for all $\alpha\in R^-$.
 \end{proof}
Thus there is no special quiver of type $E_8$. To get a special quiver $\Qu$ of one of the other simply-laced types , we are only allowed to choose vertices as sources which are framed in the following diagrams (following the classification of minuscule weights given in \cite[Chapter VIII, Proposition 7]{Bu}):

\begin{equation*}
\xymatrix{
A_n: & \fbox{\text{$\circ$}} \ar@{-}[r] &\fbox{\text{$\circ$}}\ar@{-}[r]&\fbox{\text{$\circ$}}\ar@{-}[r] & \cdots \ar@{-}[r]&\fbox{\text{$\circ$}}\\
  & & & & & & \fbox{\text{$\circ$}} \\
D_n: &\fbox{\text{$\circ$}}\ar@{-}[r] &\circ \ar@{-}[r]&\circ \ar@{-}[r]&\cdots \ar@{-}[r]& \circ\ar@{-}[ru]\ar@{-}[rd] &\\
 & & & & & & \fbox{\text{$\circ$}}
}
\end{equation*}
\begin{equation*}
\xymatrix{
& & & \circ \\
E_6 & \fbox{\text{$\circ$}} \ar@{-}[r] & \circ \ar@{-}[r]&\circ \ar@{-}[r]\ar@{-}[u] & \circ\ar@{-}[r]&\fbox{\text{$\circ$}} \\
& & & \circ \\
E_7 & \circ \ar@{-}[r] & \circ \ar@{-}[r]&\circ \ar@{-}[r]\ar@{-}[u] & \circ \ar@{-}[r]&\circ \ar@{-}[r]&\fbox{\text{$\circ$}}.
}
\end{equation*}

\section{COMPARISON}\label{comparebinfty}

In this section we give an explicit crystal isomorphism between the two crystal structures $B^{\mathscr{H}}(\infty)$ and $B^g(\infty)$. While the construction of $B^{\mathscr{H}}(\infty)$  works for isomorphism classes of $k\Qu$-modules over an arbitrary field $k$, we fix $k=\mathbb{C}$ in this section to relate it to the quiver representations appearing in the geometric construction.

We start by recalling Lusztig's description of the irreducible components of $\Lambda_V$, i.e. the elements of the crystal $B^g(\infty)$.

\begin{prop}[{\cite[Proposition 14.2.(b)]{Lu}}]\label{conormal} For $\mathfrak{g}$ semi-simple of type ADE, the irreducible components of $\Lambda_V$ are the closures of the conormal bundles of the $G_v$-orbits in $\Repp_V(\Qu).$
\end{prop}

For $M\in\Repp_V(\Qu)$, we denote the conormal bundle of the orbit $G_v\cdot M$ in $\Repp_V(\Qu)$ by ${\mathcal{C}_{{[M]}}}$. Hence, by Proposition \ref{conormal}, every irreducible component of $\Lambda_V$ is given by the closure $\overline{\mathcal{C}_{{[M]}}}$ of ${\mathcal{C}_{{[M]}}}$ for some $M$.

Since the $G_v$-orbits in $\Repp_V(\Qu)$ coincide with the isomorphism classes of representations of the path algebra $\mathbb{C}\Qu$, we have a one-to-one correspondence between the vertices of $B^g(\infty)$ and the isomorphism classes of $\mathbb{C}\Qu$-modules. Hence the map
\begin{align}\label{eq:iso}
\mathscr{F}:B^{\mathscr{H}}(\infty) & \rightarrow B^g(\infty) \\
b_{[M]} & \mapsto \overline{\mathcal{C}_{[M]}} \notag
\end{align}
is well-defined and bijective. In this section we prove that $\mathscr{F}$ is a morphism of crystals and thus provides the desired isomorphism.

Explicitly, for $M\in \Repp_V(\Qu)$, we have by \cite[Lemma 9.3]{Lu2}
$$\mathcal{C}_{[M]}=\pr^{-1}(G_v\cdot M)$$
where
\begin{align}\label{eq:pi}
\pr:\Lambda_V & \rightarrow \Repp_V(\Qu)
\end{align}
is the restriction map given by forgetting the arrows $h\notin\Qu$.
Hence $\mathscr{F}\left(b_{[M]}\right)=\overline{\pr^{-1}(G_v\cdot M)}$.

Recall from the definition of $\varepsilon_i$ given in (\ref{eq:epsilon}) and Remark \ref{dense} that, for $i\in I$, the actions of the Kashiwara operators and the value of the function $\varepsilon_i$ on $\overline{\mathcal{C}_{[M]}}$ are already determined by their values on a dense subset of the conormal bundle $\mathcal{C}_{[M]}$. The next remark shows that it suffices to study one fiber of $\mathcal{C}_{[M]}$.

\begin{rem}\label{fiber}
For $M\in \Repp_V(\Qu)$ let $\mathscr{D}\subset \pr^{-1}(M)$ be a dense subset. Then $G_v\cdot \mathscr{D}$ is a dense subset of $\overline{\mathcal{C}_{[M]}}$: Indeed by the $G_v$-equivariance of $\pr$ and the fact that each $g\in G_v$ acts as a homeomorphism, we have
$$\overline{G_v\cdot \mathscr{D}}=\overline{\displaystyle\bigcup_{g\in G_v}g\cdot \mathscr{D}}\supseteq \displaystyle\bigcup_{g\in G_v}\overline{g\cdot \mathscr{D}} \supseteq \displaystyle\bigcup_{g\in G_v}\pr^{-1}(g\cdot M)=\mathcal{C}_{[M]}.$$
\end{rem}
We proceed by recalling a description of the fiber $\pr^{-1}(M)$. Let therefore $\mathcal{G}_1,\mathcal{G}_2:\mathbb{C}\Qu-\modd\rightarrow \mathbb{C}\Qu-\modd$ be two functors. We denote by $\mathbb{C}\Qu-\modd(\mathcal{G}_1,\mathcal{G}_2)$ the following category: its objects are pairs $(M,\phi)$, where $M\in \mathbb{C}\Qu-\modd$ and $\phi\in \Hom_{\mathbb{C}\Qu}(\mathcal{G}_1(M),\mathcal{G}_2(M))$. Given two objects $(M,\phi)$, $(M',\phi')$, a morphism $(M,\phi)\rightarrow(M',\phi')$ in $\mathbb{C}\Qu-\modd(\mathcal{G}_1,\mathcal{G}_2)$ is a morphism $f:M\rightarrow M'$ in $\mathbb{C}\Qu-\modd$ such that $\phi'\circ \mathcal{G}_1(f)=\mathcal{G}_2(f) \circ \phi$.

\begin{thm}[{\cite[Theorem B., Theorem C., Proposition 3.]{Ri3}}]\label{preprojmodule}The categories $\Pi(\Qu)-\modd$,
$\mathbb{C}\Qu-\modd(\tau^{-1},\id)$ and $\mathbb{C}\Qu-\modd(\id,\tau)$ are isomorphic.
\newline
For $M\in \Repp_v(\Qu)$ we furthermore have a vector space isomorphism $\pr^{-1}(M)\cong\Hom_{\mathbb{C}\Qu}(\tau^{-1}M,M)$.
\end{thm}

\begin{rem} Via the Auslander-Reiten duality we have an isomorphism $\pr^{-1}(M)\cong \D\Ext^{1}(M,M)$, where $\D\Ext^{1}(M,M)$ is the vector space dual of the space of self-extensions of $M$.
\end{rem}

Recall from Definition \ref{special} that a Dynkin quiver is called special if $\dim\Hom_{\mathbb{C}\Qu}(X,S(i))\le 1$ for all indecomposable $\mathbb{C}\Qu$-modules $X$ and all $i\in I$.

\begin{assumption}\label{special2} {\bf From now on we assume that $\Qu$ is special. }
\end{assumption}

The proof proceeds in two step. In the first step we develop a combinatorial method and use it to show that the bijection $\mathscr{F}$ preserves, for a fixed $i\in I$, the function $\varepsilon_i$ on the crystal $B^{\mathscr{H}}(\infty)$ and $B^g(\infty)$, respectively. In the second step we use the combinatorics developed in step one to define a dense subset of $\pr^{-1}(M)$ which is then used to proof that $\mathscr{F}$ preserves the actions of the Kashiwara operators.

\subsection{EQUIVARIANCE WITH RESPECT TO \texorpdfstring{$\varepsilon_i$}{EPSILON}}
Recall from Definition \ref{definitionai} that $a_i^{\Qu}(M)=\max_{V\in \mathscr{S}_i(\Qu)}F_i(M,V)$. The main task of this step is the proof of the following.

\begin{prop}\label{3} For $M\in \Repp_v(\Qu)$ we have $a_i^{\Qu}(M)=\varepsilon_i(\overline{\mathcal{C}_{[M]}}).$
\end{prop}
First we translate $\varepsilon_i(\overline{\mathcal{C}_{[M]}})$ into a homological notion by which the function can be handled more easily in our setup. Note therefore the following equality for $\widetilde{M}\in \Lambda(v)$

\begin{equation}\label{eq:dimhomdynk}
\varepsilon_i(\widetilde{M})=\dim\Hom_{\Pi(\Qu)}(\widetilde{M},S(i)).
\end{equation}
For $M\in \mathbb{C}\Qu-\modd$ and $\phi \in \Hom_{\mathbb{C}\Qu}(\tau^{-1}M,M)$, we define
\begin{align*}
\ell_i(M)&:=\dim\Hom_{\mathbb{C}\Qu}(M,S(i)),\\
\varepsilon_{i,\phi}&:=\dim\{f\in \Hom_{\mathbb{C}\Qu}(M,S(i)) \mid f \circ \phi =0\}=\ell_i(\Coker \phi).
\end{align*}
Since $\Qu$ is special by Assumption \ref{special2}, we have \begin{equation}\label{eq:dimhom}
\ell_i(M)=\displaystyle\sum_{j\in J}\dim\Hom_{\mathbb{C}\Qu}(M_j,S(i))=\sum_{B\in \mathscr{P}_i(\Qu)}\mu_B(M)
\end{equation}
where $M=\bigoplus_{j\in J}M_j$ is the decomposition of $M$ into indecomposable direct summands. Note that for $\widetilde{M}\in \Lambda(v)$ and $g\in G_v$ we have $\varepsilon_i(\widetilde{M})=\varepsilon_i(g\cdot\widetilde{M})$. Thus from (\ref{eq:dimhomdynk}) we get with Remark \ref{fiber} and the fact that $\Ext^1(S(i),S(i))=0$ the following.

\begin{cor} We have the equality
$$\varepsilon_i(\overline{\mathcal{C}_{[M]}})=\min\{\ell_i(\Coker \phi) \mid \phi \in \Hom_{\mathbb{C}\Qu}(\tau^{-1}M,M) \}.$$
\end{cor}

For $V\in \mathscr{S}_i(\Qu)$ and $M\in \mathbb{C}\Qu-\modd$ we write in the following $M=M^{\kg V} \bigoplus M^{\ntrianglelefteq V}$, where
\begin{equation}\label{eq:abk1}
M^{\trianglelefteq V}=\displaystyle\bigoplus_{B\in \mathscr{P}_i(\Qu); B\kg V}B^{\mu_B(M)}.
\end{equation}

\begin{lem}\label{Fkleiner} For any $V\in \mathscr{S}_i(\Qu)$ and $\phi \in \Hom_{\mathbb{C}\Qu}(\tau^{-1}M,M)$ we have $$F_i(M,V)\le \varepsilon_{i,\phi}.$$
\end{lem}
\begin{proof}
First note that for $\phi \in \Hom_{\mathbb{C}\Qu}(\tau^{-1}M,M)$ the short exact sequence
$$0 \rightarrow \im\phi \rightarrow M \rightarrow \Coker \phi \rightarrow 0$$
induces the exact sequence
$$0 \rightarrow \Hom_{\mathbb{C}\Qu}(\Coker\phi,S(i)) \rightarrow \Hom_{\mathbb{C}\Qu}(M,S(i)) \rightarrow \Hom_{\mathbb{C}\Qu}(\im\phi,S(i)).$$ We thus obtain the inequality:
\begin{equation}\label{eq:inequ}
\ell_i(\Coker\phi)\ge \ell_i(M)-\ell_i(\im\phi).
\end{equation}
Setting
$$\phi_{\kg V}:=\pi_{M^{\kg}}\circ \phi,$$
where $\pi_{M^{\kg}}:M\twoheadrightarrow M^{\kg V}$ denotes the canonical projection, we have
\begin{align*}
\varepsilon_{i,\phi}&\ge \dim\{f\in\Hom_{\mathbb{C}\Qu}(M,S(i)) \mid f\circ\phi=0, \ f|_{M^{\ntrianglelefteq V}}=0\}\\
&= \dim \{f\mid\ f\circ\phi_{\kg V}=0, \ f|_{M^{\ntrianglelefteq V}}=0\} = \ell_i\left(M^{\trianglelefteq V}/\im{\phi_{\kg V}}\right).
\end{align*}
Since, by (\ref{eq:dimhom})
\begin{align*}
\ell_i(M^{\trianglelefteq V})&=\underset{B\kg V}{\sum_{B\in\mathscr{P}_i(\Qu),}}\mu_B(M)\,\,\,\,\,\,\text{ and} \\
\ell_i(\im{\phi}_{\kg V})&\le \ell_i\left((\tau^{-1}M)^{\trianglelefteq V}\right)=\underset{B\kg V}{\sum_{B\in\mathscr{P}_i(\Qu),}}\mu_B(\tau^{-1}M)
\end{align*}
we obtain by (\ref{eq:inequ}) applied to $\phi_{\kg V}$:
$$\ell_i\left(M^{\trianglelefteq V}/\im{\phi_{\kg V}}\right) \ge \ell_i(M^{\trianglelefteq V})-\ell_i(\im{\phi_{\kg V}}) \ge \underset{B\kg V}{\sum_{B\in\mathscr{P}_i(\Qu),}}\mu_B(M)-\mu_{B}\left(\tau^{-1}M\right) = F_i(M,V).
$$
\end{proof}

We write $M^{\kg S(i)}=\bigoplus_{j=1}^{m_1}B_j$ and $(\tau^{-1}M)^{\kg S(i)}=\bigoplus_{k=1}^{m_2}C_k$ for the  decompositions into indecomposable direct summands. We call $\phi \in \Hom_{\mathbb{C}\Qu}(\tau^{-1}M, M)$ \emph{combinatorial} if
\begin{itemize}
\item $\phi|_{(\tau^{-1}M)^{\ntrianglelefteq S(i)}}=0$,
\item $\phi|_{C_k}$ is either zero or $\im(\phi|_{C_k})\subset B_{j(k)}$ for a $j(k)\in\{1,2,\ldots,m_1\}$ such that $\phi|_{C_k}$ is inducing an isomorphism $\Hom_{\mathbb{C}\Qu}(B_{j(k)},S(i))\cong \Hom_{\mathbb{C}\Qu}(C_k,S(i))$,
\item $j(k_1)\ne j(k_2)$ for $k_1\ne k_2$.
\end{itemize}

\begin{lem}\label{combinatorial}  Assume that $\phi \in \Hom_{\mathbb{C}\Qu}(\tau^{-1}M, M)$ is combinatorial. Then
\begin{equation}\label{additive}\ell_i(\Coker\phi)=\ell_i(M)-\ell_i(\im\phi)
\end{equation}
and
$$\ell_i(\im\phi)=\#\{B\in \mathscr{P}_i(\Qu) \mid \mu_B(\tau^{-1}M) \ne 0 \text{ and } \phi(B)\ne 0\}.$$
\end{lem}

\begin{proof}
Since by assumption that $\phi$ is combinatorial for any $k$ the map
\begin{equation*} \Hom_{\mathbb{C}\Qu}(B_{j(k)},S(i))\rightarrow\Hom_{\mathbb{C}\Qu}(\phi \left( C_k \right),S(i))\end{equation*}
is surjective, we obtain that the  map $\psi$ in the exact sequence
$$0\rightarrow \Hom_{\mathbb{C}\Qu}(\Coker\phi,S(i)) \rightarrow \Hom_{\mathbb{C}\Qu}(M,S(i)) \xrightarrow{\psi} \Hom_{\mathbb{C}\Qu}(\im\phi,S(i))$$
is surjective. This implies (\ref{additive}). Furthermore we have, since $\Qu$ is special,
$$\ell_i(\im\phi)=\sum_{k:\phi(C_k)\ne 0}\ell_i(C_k)=\#\{B\in \mathscr{P}_i(\Qu) \mid \mu_B(\tau^{-1}M) \ne 0 \text{ and } \phi(B)\ne 0\}.$$

\end{proof}

We now fix $i\in I$. To prove Proposition \ref{3}, we show that for each $\phi$ in a certain class of combinatorial homomorphisms there exists $V^{\phi} \in \mathscr{S}_i(\Qu)$ such that $F_i(M,V^{\phi})=\varepsilon_{i,\phi}$. For this we work in a purely combinatorial setup by introducing the directed graph $\mathscr{P}_M^{\infty}$. This graph, which corresponds to the arrangement of direct summand of $M$ and $\tau^{-1}M$ in the poset $\mathscr{P}_i(\Qu)$, forms the main tool of our approach.

To construct $\mathscr{P}_M^{\infty}$, let $\mathscr{P}$ be the Hasse diagram corresponding to the poset $\mathscr{P}_i(\Qu)$, i.e. there is a vertex $v_B$ in $\mathscr{P}$ corresponding to each $B\in\mathscr{P}_i(\Qu)$ and an arrow $v_{B_1}\rightarrow v_{B_2}$ in $\mathscr{P}$ if and only if $B_1\vartriangleleft B_2$ minimally. For each $\mathbb{C}\Qu$-module $M$, we construct the following oriented graph $\mathscr{P}_M^{\infty}$:

We replace each vertex $v_{B}$ of $\mathscr{P}$ by a chain
$$v_{B^{(1)}}\rightarrow v_{B^{(2)}}\rightarrow\ldots\rightarrow v_{B^{(l_B)}}$$ where $l_B:=\max(1, \mu_B(M),\mu_B(\tau^{-1}M)).$ For each arrow $v_{B_1}\rightarrow v_{B_2}$ in $\mathscr{P}$, we add an arrow $v_{{B_1}^{(l_{B_1})}}\rightarrow v_{{B_2}^{(1)}}$ in $\mathscr{P}_M^{\infty}$. We further add a vertex $v_{\infty}$ and for each vertex $B$ of $\mathscr{P}$ which corresponds to a $\kg$-maximal element of $\mathscr{P}_i(\Qu)$, we add an arrow $v_{B^{(l_B)}}\rightarrow v_{\infty}$ in $\mathscr{P}_M^{\infty}$.

\begin{ex}\label{kgminimal} Let $\Qu= 1 \leftarrow 2 \rightarrow 3$ and $i=2$
(compare with part \ref{Qalternatierend} of Example \ref{15}). Let $[M]= [111]\oplus [111] \oplus [100] \oplus [011] \oplus [010]$, then $\mathscr{P}_M^{\infty}$ looks as follows.
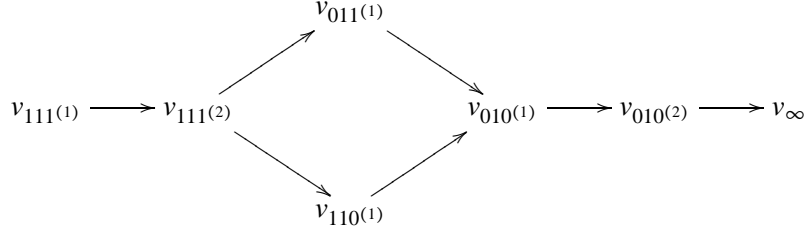
\begin{equation*}
\xymatrix{
& & v_{011^{(1)}} \ar@{->}[rd] & & \\
v_{111^{(1)}} \ar@{->}[r] & v_{111^{(2)}} \ar@{->}[ru]\ar@{->}[rd] & & v_{010^{(1)}} \ar@{->}[r] & v_{010^{(2)}} \ar@{->}[r] & v_{\infty}\\
& &  v_{110^{(1)}} \ar@{->}[ru] & & &}
\end{equation*}

\end{ex}

Let $(\mathscr{P}_M^{\infty})_0$ be the set of vertices of the graph $\mathscr{P}_M^{\infty}$. We extend the partial order on $\mathscr{P}$ to a partial order on $(\mathscr{P}^{\infty}_M)_0$ by setting $v_1\kg v_2$ if and only if there is a path from $v_1$ to $v_2$ in $\mathscr{P}^{\infty}_M$.

For calculations in our graph, we introduce the category $\mathscr{A}\equiv\mathscr{A}_{\mathscr{P}_M^{\infty}}$ in which the objects are subsets of $(\mathscr{P}_M^{\infty})_0$ . For $W_1, W_2 \in \mathscr{A}$, a morphism $\phi:W_1 \rightarrow W_2$ is a map of sets $\phi: W_1\cup v_{\infty} \rightarrow W_2\cup v_{\infty}$ satisfying the following properties:
\begin{enumerate}
\item for all $w \in W_1$, we have $w\le \phi(w)$,
\item $\phi|_{W_1 \backslash \left\{\phi^{-1}(v_{\infty})\right\}}$ is injective.
\end{enumerate}

\begin{rem}
The first defining property of a morphism in $\mathscr{A}$ corresponds to the fact that we want to resemble homorphisms in $\mathbb{C}\Qu-\modd$ and the second defining property comes from the fact that we are interested in the study of combinatorial $\phi\in \Hom_{\mathbb{C}\Qu}(\tau^{-1}M,M)$.
\end{rem}

We color the subset of $(\mathscr{P}^{\infty}_M)_0$ corresponding to the direct summands of $\tau^{-1} M$ red and the subset of $(\mathscr{P}^{\infty}_M)_0$ corresponding to the direct summands of $M$ white and call these sets $R$ and $W$, respectively. Using these two sets we introduce combinatorial analogs of the functions $\varepsilon_{i,\phi}$  and $F_i(M,V)$.

For a subset $V\subset (\mathscr{P}^{\infty}_M)_0$ we denote by $|V|$ the cardinality of this set an we define
$$F_i(V):=| W\cap V | - | R\cap V |.$$ For any $\phi \in \Hom_{\mathscr{A}}(R,W)$ let further
$$\varepsilon_i(\phi):=| W\backslash \phi(R)|.$$
We define the closure of $V$with respect to the ordering $\kg$.
$$V^{\downarrow}:=\{P\in ({\mathscr{P}_M^{\infty}})_0 \mid P \kg v_B \text{ for some }v_B\in V\}.$$
A crucial role in our approach plays the following preorder on $\Hom_{\mathscr{A}}(R,W)$:
$$\phi \preceq \psi \text{ if and only if } \exists \rho \in \Hom_{\mathscr{A}}(W,W) \text{ and } \psi(R)=\rho\circ \phi(R),$$
where the equality is an equality of sets.

Loosely speaking, $\phi \preceq \psi$ says that we can move the elements of $\phi(R)$ to the elements of $\psi(R)$ along paths in $\mathscr{P}_M^{\infty}$.

\begin{ex}
We continue with Example \ref{minimal}. Here $W=\{v_{111^{(1)}},v_{111^{(2)}},v_{110^{(1)}},v_{011^{(1)}},v_{101^{(1)}}\}$ and $R=\{v_{010^{(1)}},v_{010^{(2)}}\}$. Let $\phi_1,\phi_2 \in \Hom_{\mathscr{A}}(R,W)$ be given by
\begin{align*}
\phi_1(v_{010^{(1)}})&=v_{010^{(1)}},  & \phi_1(v_{010^{(2)}})&=v_{\infty},\\
\phi_2(v_{010^{(1)}})&=v_{\infty}, & \phi_2(v_{{010}^{(2)}})&=v_{\infty}.
\end{align*}
For $\rho\in\Hom_{\mathscr{A}}(W,W)$ given by $\rho(v_{010^{(1)}})=v_{\infty}$ and $\rho|_{W\backslash\{ v_{010^{(1)}}\}}=\id_{{W\backslash \{v_{010^{(1)}}\}}}$, we have $\rho \circ \phi_1=\phi_2$ and thus $\phi_1 \preceq \phi_2$.
\end{ex}

We define $\phi\in \Hom_{\mathscr{A}}(R,W)$ to be \emph{$\preceq$-minimal} if for each $\psi\in \Hom_{\mathscr{A}}(R,W)$ such that $\psi \preceq \phi$, we also have $\phi \preceq \psi$.

Note that this ordering is not anti-symmetric, so the above does not imply $\phi=\psi$. The notion of $\preceq$-minimality allows us to prove the following proposition.

\begin{prop}\label{f=e} Let $\phi\in \Hom_{\mathscr{A}}(R,W)$ be $\preceq$-minimal and assume that $\varepsilon_i(\phi)>0$. Then there exists $V^{\phi} \subset \mathscr{P}_M^{\infty}\backslash \{v_{\infty}\}$ such that
$$F_i(V^{\phi})=\varepsilon_i(\phi).$$
Furthermore we have $W\backslash V^{\phi} \subset \im \phi$ .
\end{prop}

\begin{proof}
Let $\phi\in \Hom_{\mathscr{A}}(R,W)$ and $V\subset  \mathscr{P}_M^{\infty}\backslash \{v_{\infty}\}$. Note that we have the following inequalities

\begin{align*}
F_i(V^{\downarrow}) &= |W\cap V^{\downarrow}| - | R\cap V^{\downarrow} |  \overset{\text{($\star$)}}{\le} | W \cap V^{\downarrow} | - | \phi\left(R\right) \cap V^{\downarrow}|  \\
  &  \overset{\text{($\star\star$)}}{=} | W \backslash \phi(R) \cap V^{\downarrow} |
 \overset{\text{($\star\star\star$)}}{\le} | W \backslash \phi(R) | = \varepsilon_i(\phi),
\end{align*}
where the inequality ($\star$) comes from the first defining property of a morphism in $\mathscr{A}$ and the equality ($\star\star$) comes from the assumption that $v_{\infty}\notin V$. Note that $\phi(R)\cap V^{\downarrow}=\phi(R\cap V^{\downarrow})\cap V^{\downarrow}$ by the first defining property of a morphism in $\mathscr{A}$. Thus, by the second defining property of a morphism in $\mathscr{A}$, the inequality $(\star)$ is an equality if and only if

\begin{equation}\label{eq:stabel}
\phi\left(R \cap V^{\downarrow}\right) \subset V^{\downarrow}.
\end{equation}
Further the inequality $(\star\star\star)$ is an equality if and only if

\begin{equation}\label{eq:stabel2}
W\backslash \phi(R) \subset V^{\downarrow}.
\end{equation}
Let us now further assume that $\varepsilon_i(\phi)>0$ which implies $W\backslash \phi(R)\ne\varnothing$. We extend $\left(W\backslash \phi(R)\right)^{\downarrow}$ to a subset of $(\mathscr{P}_M^{\infty})_0$ satisfying Property (\ref{eq:stabel}). For that, let $\mathcal{P}(({\mathscr{P}_M^{\infty}})_0)$ be the power set of $({\mathscr{P}_M^{\infty}})_0$ and consider the operator
\begin{align*} \Phi: \mathcal{P}(({\mathscr{P}_M^{\infty}})_0) & \rightarrow \mathcal{P}(({\mathscr{P}_M^{\infty}})_0)\\
V &\mapsto\left(\phi \left(R \cap V^{\downarrow}  \right) \cup V \right)^{\downarrow}.
\end{align*}
Note that, for $V_1,V_2\subseteq (\mathscr{P}_M^{\infty})_0$ with $V_1\subset V_2$, we have
\begin{equation}\label{eq:monotony}
V_1\subset \Phi(V_1) \subset \Phi(V_2).
\end{equation}
We therefore obtain the closure operator
\begin{align*}  \HH_\phi:  \mathcal{P}(({\mathscr{P}_M^{\infty}})_0) &\rightarrow  \mathcal{P}(({\mathscr{P}_M^{\infty}})_0) \\
\HH_{\phi}(V)=&\Phi^{k}(V),
\end{align*}
where $k$ is the smallest number such that $\Phi^{k}(V)=\Phi^{k +1}(V)$. We then define $V^{\phi}\subset (\mathscr{P}_M^{\infty})_0$ by
\begin{equation}\label{eq:huellenv}
V^{\phi}:=\HH_\phi (W\backslash \phi(R)).
\end{equation}
We obtain that $V^\phi$ satisfies by construction $V^\phi=\left(V^\phi\right)^{\downarrow}$ as well as property (\ref{eq:stabel}) and (\ref{eq:stabel2}). Hence we have the claimed equality $F_i(V^{\phi})=\varepsilon_i(\phi)$ if and only if $v_{\infty}\notin V^{\phi}$. Note further that since $W\backslash\phi(R)\subset V^{\phi}=\HH_{\phi}(W\backslash\phi(R))$, we have $W\backslash V^{\phi}\subset \im\phi$.

Assume that $\phi$ is $\preceq$-minimal. It remains to show that $v_{\infty}\notin V^{\phi}$.

If $v_{\infty}\in V^{\phi}$ then $v_{\infty}\in \Phi^k(W\backslash \phi(R))$ but $v_{\infty}\notin \Phi^{k-1}(W\backslash \phi(R))$ since $v_{\infty}^{\downarrow}=(\mathscr{P}_M^{\infty})_0$. Furthermore, since $v_{\infty}\in \Phi^k(W\backslash \phi(R))$, there exists $r_k \in \left(\Phi^{k-1}(W\backslash \phi(R))\right)\cap R$ such that $\phi(r_k)=v_{\infty}$. Likewise, since $r_k \in \left(\Phi^{k-1}(W\backslash \phi(R))\right)\cap R$, there exists $r_{k-1}\in \left(\Phi^{k-2}(W\backslash \phi(R))\right)\cap R$ such that $\phi(r_{k-1})=r_k$. Repeating this argument we get $r_1,r_2,\ldots, r_k \in W\cap R$ such that $\phi(r_j)=r_{j+1}$ for $1 \le j \le k-1$. But this is only possible if there exists $w\in W\backslash \phi(R)$ with $r_1 \kg w$ (otherwise $\phi(r_1) \in \left(W\backslash \phi(R)\right)^{\downarrow}$).

We define $\phi'\in\Hom_{\mathscr{A}}(R,W)$ by
\begin{align*}
\phi' |_{R\backslash \{r_1,r_2,\ldots r_k\}}&=\phi|_{R\backslash \{r_1,r_2,\ldots r_k\}},\\
\phi'(r_j)&=r_j \text{ for }1 <  j \le k-1\\
\phi'(r_1)&=w.
\end{align*}
Then $\rho \circ \phi'(R)=\phi(R)$ for $\rho \in \Hom_{\mathscr{A}}(W,W)$ given by $\rho(w)=v_{\infty}$ and $\rho_{W\backslash\{w\}} = \id_{W\backslash\{w\}}$.
By $\preceq$-minimality of $\phi$, there exists $\rho' \in \Hom_{\mathscr{A}}(W,W)$ with $\im \rho'\circ \phi  = \im \phi'$, yielding
$$
\im \phi'  = \im \underbrace{\rho' \circ \rho}_{=:\widetilde\rho} \circ \phi'.
$$
This implies that $\widetilde\rho$ and hence $\rho|_{\im \phi'}$ is injective in contradiction to $\rho (w) = \rho (v_{\infty})=v_{\infty}$. The equality $\rho (v_{\infty})=v_{\infty}$ here comes from the first defining property of a morphism in $\mathscr{A}$.
\end{proof}

\begin{ex} We give an example for the construction of $V^{\phi}$ for a minimal $\phi\in\Hom_{\mathscr{A}}(R,W)$. Assume that $\mathscr{P}_M^{\infty}$ is given as follows:
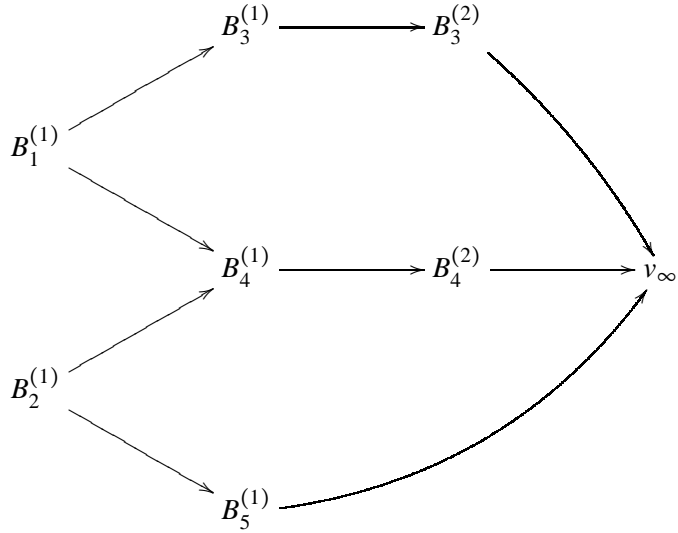
\begin{equation*}
\xymatrix{  & & & & & \\
& &{B_3^{(1)}} \ar[rr]   &  & {B_3^{(2)}}\ar@/{}^{.5pc}/[rrdd] \\
{B_1^{(1)}}   \ar[rru]  \ar[rrd] & & &  &  \\
& & {B_4^{(1)}} \ar[rr] &  & {B_4^{(2)}} \ar[rr]  & & v_{\infty}\\
{B_2^{(1)}}  \ar[rru] \ar[rrd]& & & & \\
& &{B_5^{(1)}}\ar@/{}_{1.8pc}/[rrrruu] & & &
}
\end{equation*}
Further assume that $R$ and $W$ are given as follows:
\begin{align*}
R&=\{v_{{B_1}^{(1)}}, v_{{B_2}^{(1)}}\} \text{ and }\\ W&=\{v_{{B_3}^{(1)}},v_{{B_3}^{(2)}},v_{{B_4}^{(1)}},v_{{B_4}^{(2)}},v_{{B_5}^{(1)}}\}.
\end{align*}
We define $\phi\in\Hom_{\mathscr{A}}(R,W)$ by
\begin{align*} \phi(v_{{B_1}^{(1)}})&=v_{{B_4}^{(1)}},\\ \phi(v_{{B_2}^{(1)}})&=v_{{B_4}^{(2)}}
\end{align*}
and note that $\phi$ is $\preceq$-minimal. We illustrate the situation in Picture (\ref{eq:redandwhitepoints}) below where the vertices in $W$ are drawn as circles and the vertices in $R$ are drawn as red bullets. The dotted lines indicate how the vertices in $R$ are mapped by $\phi$.

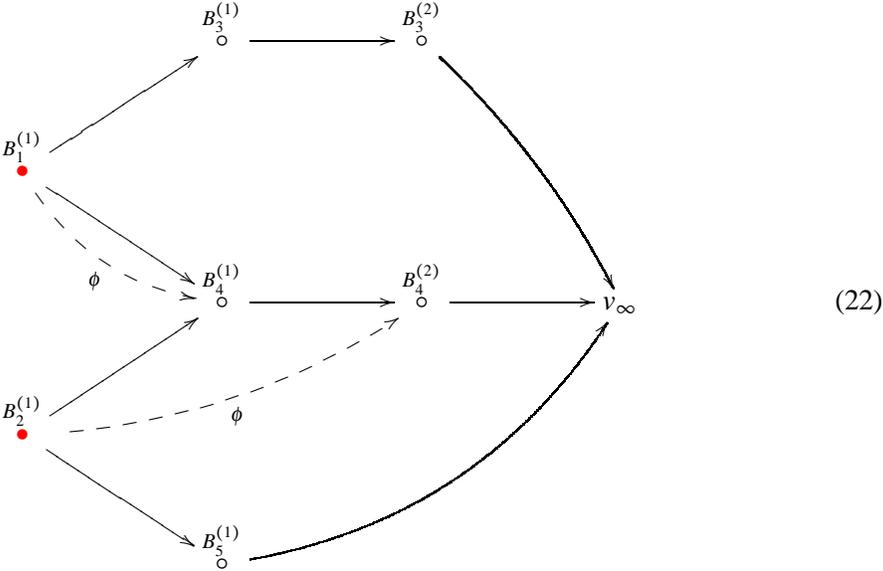
\begin{equation}\label{eq:redandwhitepoints}
\begin{gathered}
\xymatrix{
 & & & & &  \\
& & \overset{B_3^{(1)}}{\circ} \ar[rr]   &  & \overset{B_3^{(2)}}{\circ} \ar@/{}^{.5pc}/[rrdd] \\
\overset{B_1^{(1)}}{\textcolor{red}{\text{$\bullet$}}}  \ar@/_{.99pc}/@{-->}[rrd]_{\phi} \ar[rru]  \ar[rrd] & & & &  \\
&  & \overset{B_4^{(1)}}{\circ} \ar[rr] &  & \overset{B_4^{(2)}}{\circ} \ar[rr]  & & v_{\infty}\\
\overset{B_2^{(1)}}{\textcolor{red}{\text{$\bullet$}}}   \ar@/_{.99pc}/@{-->}[rrrru]_{\phi} \ar[rru] \ar[rrd]& & & & \\
& & \overset{B_5^{(1)}}{\circ}\ar@/{}_{1.8pc}/[rrrruu] & & \\
& & & &
}
\end{gathered}
\end{equation}

We have
$$W\backslash \phi(R)=\{v_{{B_3}^{(1)}},v_{{B_3}^{(2)}},v_{{B_5}^{(1)}}\}.$$

Now the operator $\Phi$ adds to the set $W\backslash \phi(R)$ the closure (with respect to $\kg$) of all vertices that are in the image of $\phi |_{v_B}$ for each $v_B \in R$ with $v_B \kg W\backslash \phi(R)$. Hence the closure of the image of $\phi$ restricted to the set of all $v_B$ that are on the left of the blue dotted lines in Picture (\ref{eq:dotted}).

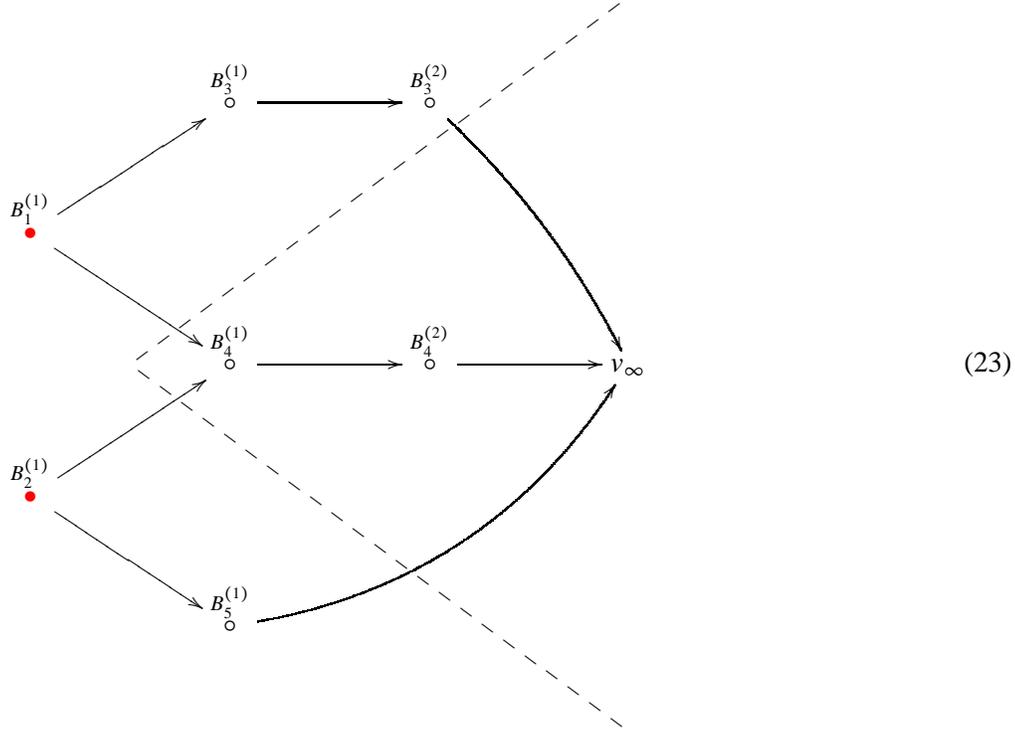
\begin{equation}\label{eq:dotted}
\begin{gathered}
\xymatrix{  & & & & & & \ar@{--}@[blue][dddlllll] \\
& & \overset{B_3^{(1)}}{\circ} \ar[rr]   &  & \overset{B_3^{(2)}}{\circ} \ar@/{}^{.5pc}/[rrdd] \\
\overset{B_1^{(1)}}{\textcolor{red}{\text{$\bullet$}}}   \ar[rru]  \ar[rrd] & & &  &  \\
& \ar@{--}@[blue][dddrrrrr] & \overset{B_4^{(1)}}{\circ} \ar[rr] &  & \overset{B_4^{(2)}}{\circ} \ar[rr]  & & v_{\infty}\\
\overset{B_2^{(1)}}{\textcolor{red}{\text{$\bullet$}}}   \ar[rru] \ar[rrd]& & & & \\
& & \overset{B_5^{(1)}}{\circ}\ar@/{}_{1.8pc}/[rrrruu] &  &   &  & & \\
& & & & & & & & &
}
\end{gathered}
\end{equation}
Thus
\begin{align*}
\Phi(W\backslash \phi(R))&=\left(\phi((W\backslash \phi(R))^{\downarrow}\cap R )\cup (W\backslash \phi (R))\right)^{\downarrow} \\
&=\{v_{{B_1}^{(1)}}, v_{{B_2}^{(1)}},v_{{B_3}^{(1)}},v_{{B_3}^{(2)}},v_{{B_4}^{(1)}},v_{{B_4}^{(2)}},v_{{B_5}^{(1)}}\},\\
\Phi^2(W\backslash \phi(R))&=\Phi (W\backslash \phi(R)).
\end{align*}
We note that the $\Phi$ already stabilizes after being applied the first time and does thus coincide with the closure operator $\HH_{\phi}$.

We conclude
$$V^{\phi}=\{v_{{B_5}^{(1)}},v_{{B_4}^{(2)}},v_{{B_3}^{(2)}}\}^{\downarrow}.$$
\end{ex}

We are now able to prove Proposition \ref{3}:
\begin{proof}[Proof of Proposition \ref{3}]

Since Lemma \ref{Fkleiner} yields $F_i(M,V)\le \varepsilon_{i,\psi}$ for any $V\in \mathscr{S}_i(\Qu)$ and any
$\psi \in \Hom_{\mathbb{C}\Qu}(\tau^{-1}M,M)$, it suffices to show the existence of a $\phi\in \Hom_{\mathbb{C}\Qu}(\tau^{-1}M,M)$ and a $V$ in $\mathscr{S}_i(\Qu)$ such that $F_i(M,V)= \varepsilon_{i,\psi}$.

Let $\phi$ be any $\preceq$-minimal element in
$\Hom_{\mathscr{A}}(R,W)$. We choose a corresponding element in $\Hom_{\mathbb{C}\Qu}(\tau^{-1}M,M)$, which we denote by $\widetilde{\phi}$, the following way:
for any $B,B'\in \mathscr{P}_i(\Qu)$ such that $v_{B}\in R$, $v_{B'}\in W$ and $\phi(v_B)=v_{B'}$, we let $\widetilde{\phi}|_B:B \rightarrow M$ be a composition of irreducible morphisms $B\rightarrow B'$ that induces an isomorphism $\Hom_{\mathbb{C}\Qu}(B',S(i))\rightarrow \Hom_{\mathbb{C}\Qu}(B,S(i))$ which exists by Proposition \ref{partialorder}. Then $\widetilde\phi$ is combinatorial and we get by Proposition \ref{combinatorial}

$$\ell_i\left(\Coker\widetilde{\phi}\right)=\displaystyle\sum_{B\in \mathscr{P}_i(\Qu)} \mu_B(M) - \#\{B\in \mathscr{P}_i(\Qu) \mid \mu_{(B)}(\tau^{-1} M)\ne 0 \text{ and } \widetilde{\phi}(B)\ne 0 \}.$$

This yields the equality

$$\varepsilon_i(\phi)= W\backslash \phi(R) = \varepsilon_{i,\widetilde{\phi}}.$$

Let us first assume that $\varepsilon_{i,\widetilde{\phi}}>0$.

By Proposition \ref{f=e} we deduce that  $V^{\phi}=\HH_{\phi}(W\backslash(\phi(R)))$ satisfies the equality
$$\varepsilon_i(\phi)=F_i(V^{\phi}).$$

Since $V^{\phi} \in \left(\mathscr{P}_M^{\infty}\right)_0 \backslash \{v_{\infty}\}$, we get an induced element in $\mathscr{S}_i(\Qu)$ by taking the direct sum of those $B\in \mathscr{P}_i(\Qu)$ for which $v_B$ is a $\kg$-maximal element of $V^{\phi}$. We denote the induced element in $\mathscr{S}_i(\Qu)$ also by $V^{\phi}$ by abuse of notation. This yields the claim for this case since we have an equality $F_i(M,V^{\phi})=F_i(V^{\phi})$.

Now assume $\varepsilon_{i,\widetilde{\phi}}=0$. Let $B\in\mathscr{P}_i(\Qu)$ be the $\kg$-minimal element. Then $B$ is projective (see Remark \ref{minimal}) which implies there cannot be a direct summand $C$ of $\tau^{-1}M$ such that $C\kg B$. Note further that $B$ cannot be a direct summand of $M$: Otherwise  $\left(W\backslash (\phi(R))\right)\ne\varnothing$. We set $V_0=B$ and note that
$$0=F_i(M,V_0)=\varepsilon_{i,\phi}$$
which finishes the proof.

\end{proof}

We say that $\phi\in \Hom_{\mathbb{C}\Qu}(\tau^{-1}M,M)$ is \emph{$\preceq$-minimal} if $\phi$ is combinatorial and induces a $\preceq$-minimal $\phi\in \Hom_{\mathscr{A}}(R,W)$.

\begin{lem}\label{V0} Fix $M\in \mathbb{C}\Qu-\modd$ such that $a_i^{\Qu}(M)>0$. For any $\preceq$-minimal $\phi\in\Hom_{\mathbb{C}\Qu}(\tau^{-1}M,M)$, $V^{\phi}$ is the unique $\kg$-minimal element of $\mathscr{S}_i(\Qu)$ such that $F_i(M,V)$ is maximal.
\end{lem}

\begin{proof} Let $V$ be any element in $\mathscr{S}_i(\Qu)$ such that $F_i(M,V)$ is maximal. Then $F_i(M,V)=F_i(M,V^{\phi})=\varepsilon_i(\phi)$. Now the proof is a Corollary of the arguments used in the proof of Proposition \ref{f=e}. Namely, as already deduced there, for any
$V\in \mathscr{S}_i(\Qu)$, we have
\begin{equation*}
F_i(V^{\downarrow}) = | W\cap V^{\downarrow} | - | R \cap V^{\downarrow} | \leq  | W\backslash (\phi(R) \cap V^{\downarrow}) |
\leq | W\backslash \phi(R) |
=\varepsilon_i(\phi),
\end{equation*}
where the first inequality is an equality if and only if $\phi(R\cap V^{\downarrow})\subset (W\cap V^{\downarrow})$ and the second inequality is an equality if and only if $(W\backslash \phi(R)) \subset V^{\downarrow}$.  By construction of the closure operator $\HH_{\phi}$, for any $V\in\mathscr{S}_i(\Qu)$ satisfying those properties, we have
$$(V^{\phi})^{\downarrow}\subset \HH_{\phi}(V^{\downarrow})=V^{\downarrow}.$$
Thus $V^{\phi}\kg V$.
\end{proof}

Consequently, the element $V^{\phi}\in\mathscr{S}_i(\Qu)$, defined in (\ref{eq:huellenv}), does not depend on the choice of a $\preceq$-minimal $\phi\in \Hom_{\mathscr{A}}(R,W)$. We are thus able to define for any $\preceq$-minimal $\phi\in\Hom_{\mathscr{A}}(R,W)$
$$V^M=H_\phi(W\backslash \phi(R)).$$

We remark that this is an alternative way to prove that there is a unique $V\in \mathscr{S}_i(\Qu)$ which is $\kg$-minimal such that $F_i(M,V)=a_i^{\Qu}(M)$.

We conclude this step with a Lemma that is needed in step three.

\begin{lem}\label{weissepunkte} For each indecomposable direct summand $B$ of $V^M$ there exists $\phi \in \Hom_{\mathscr{A}}(R,W)$ with $\varepsilon_i(\phi)=a_i^{\Qu}(M)$ and $v_B\in W\backslash \phi(R)$.

In particular, there exists a combinatorial $\psi\in \Hom_{\mathbb{C}\Qu}(\tau^{-1}M,M)$ with $\varepsilon_{i,\psi}=a_i^{\Qu}(M)$ such that $B$ is a direct summand of $\Coker\psi$.
\end{lem}
\begin{proof} Let $B$ be a direct summand of $V^M$ and $\phi\in\Hom_{\mathscr{A}}(R,W)$ be any $\preceq$-minimal morphism. By construction we have that $v_B$ is a $\kg$-maximal element of $H(W\backslash \phi(R))$. If $v_B\in W\backslash\phi(R)$ we are done. Otherwise, by the definition of the operator $\HH_{\phi}$ there exists $r_1,r_2, \ldots, r_j\in R$ and a $\kg$-maximal element $w \in \HH_{\phi}(W\backslash \phi(R))$ such that $\phi(r_1)=v_B$, $r_{k-2}\kg \phi(r_{k-1})$ for all $3 \le k \le j$ and $r_j \kg w$.
We define $\phi_1\in\Hom_{\mathscr{A}}(R,W)$ by
\begin{align*}
\phi_1 |_{R\backslash \{r_1,r_2,\ldots r_{j}\}}&=\phi|_{R\backslash \{r_1,r_2,\ldots r_{j}\}},\\
\phi_1(r_k)&=\phi(r_{k+1}) \text{ for all }1 \le k \le j-1\\
\phi_1(v_j)&=w.
\end{align*}
Thus $v_B \in W\backslash \phi_1(R)$ and $$\varepsilon_i(\phi_1)=\varepsilon_i(\phi)=a_i^{\Qu}(M).$$
Analog to the proof of Proposition \ref{3} we choose a corresponding combinatorial homomorphism $\psi\in \Hom_{\mathbb{C}\Qu}(\tau^{-1}M,M)$ which yields the last statement of the Lemma.
\end{proof}

\subsection{EQUIVARIANCE WITH RESPECT TO THE KASHIWARA OPERATORS}

We fix $i\in I$ and $M\in \mathbb{C}\Qu-\modd$ with $a_i^{\Qu}(M)>0$. Recall from Definition \ref{operatorm} the operator
$$\mi M:=N\oplus \bigoplus_{B\in l(V_0)}\tau B,$$
where $M=N\bigoplus V_0$ with $V_0$ the $\kg$-minimal element of $\mathscr{S}_i(\Qu)$ such that $F_i(M,V)$ is maximal and $l(V_0)=\{B\in \mathscr{P}_i(\Qu) \mid B \ntrianglelefteq V_0 \text{ minimally}\}$.

Recall further from Proposition \ref{homologicalkashiwara} that $\tilde{e}_i b_{[M]}=b_{[\mi M]}$ for $b_{[M]}\in B^{\mathscr{H}}(\infty)$. In this section we show that bijection $\mathscr{F}$ given in (\ref{eq:iso}) preserves the Kashiwara operator $\tilde{e}_i$ and conclude that $\mathscr{F}$ is an isomorphism of crystals.

By Lemma \ref{V0}, we have $V_0\cong V^{M}$. We further write,
$$U^{M}:=\bigoplus_{B\in l(V_0)}\tau B.$$
For a direct summand $N$ of $M$, we denote by $\pi_N:M\twoheadrightarrow N$ the canonical projection and by $\iota_N: N\hookrightarrow M$ the canonical inclusion. We introduce the following notion.

\begin{defi}\label{descends} A homomorphism $\phi\in \Hom_{\mathbb{C}\Qu}(\tau^{-1}M,M)$ \emph{descends via $V^M$} if and only if there exists a short exact sequence
$$ 0 \rightarrow \mi M \xrightarrow{\iota} M \xrightarrow{f} S(i) \rightarrow 0$$
such $f\circ \phi =0$ and $\pi_N\circ\iota|_N=\id_N$.
\end{defi}
\begin{rem} The condition $\pi_N\circ\iota|_N=\id_N$ in Definition \ref{descends} is of technical nature. It is used to simplify the induction step in the proof of Proposition \ref{dichtemenge}.
\end{rem}
We decompose $N=N^{+}\oplus N^{-}$, where $N^{-}=N^{\kg V^M}$ (compare with (\ref{eq:abk1})) and let
$$V^M=\bigoplus_{k\in \mathscr{V}} V_k$$
be a decomposition of $V^M$ into indecomposable direct summands with $\mathscr{V}$ the corresponding index set. For $j\in \mathscr{V}$, we write $V^M=V_j^{\perp}\bigoplus V_j$. We abbreviate further (see again (\ref{eq:abk1}))
\begin{align*}
M^{\kg}&:=M^{\kg V^M}, &
M^{\ntrianglelefteq}&:=M^{\ntrianglelefteq V^M},\\
(\tau^{-1}M)^{\kg}&:=(\tau^{-1}M)^{\kg V^M}, &
(\tau^{-1}M)^{\ntrianglelefteq}&:=(\tau^{-1}M)^{\ntrianglelefteq V^M}.
\end{align*}
Note that $M^{\kg}=N^{-}\oplus V^M$.

\begin{lem}\label{krit} Let $\phi \in \Hom_{\mathbb{C}\Qu}(\tau^{-1}M,M)$. Assume that there exists $0\ne f\in \Hom_{\mathbb{C}\Qu}(M,S(i))$ such that
\begin{equation*}
\begin{matrix}
f\circ \phi =0, &
f|_{N^+}=0 & \text{ and } &
f|_{V_k}\ne 0 \text{ for all }k\in \mathscr{V}.
\end{matrix}
\end{equation*}
Then $\phi$ descends via $V^{M}.$
\end{lem}

\begin{proof}
By Theorem \ref{U}, we have a short exact sequence
\begin{equation*}
\xymatrix{
0 \ar[r] & U^M \ar[r]^{\iota_0}  & V^M \ar[r]^{f|_{V^{M}}} & S(i) \ar[r] &0 }
\end{equation*}

Let $N^-=\bigoplus_{k\in\mathscr{N}}B_k$ be a decomposition into indecomposable direct summands. For each $k\in \mathscr{N}$ there exists, by definition of $N^-$, a $j\in \mathscr{V}$ such that $B_k\kg V_j$. Thus, by Proposition \ref{partialorder}, there exists $\psi^k_{j}\in\Hom_{\mathbb{C}\Qu}(B_k,V_j)$ and $\lambda \in \mathbb{C}$  such that
$$f|_{B_k}=\lambda f|_{V_j}\psi^k_{j}.$$

Let $\iota:\mi M \rightarrow M$ be the following homomorphism:
\begin{equation*}
\begin{matrix}
\iota|_{U^{M}}=\iota_0, &
\iota|_{N}=\iota_{N}, &
\iota|_{B_k}=\iota_{B_k}-\lambda\psi^k_{j}
\end{matrix}
\end{equation*}
for each $k\in\mathscr{N}$.

This yields the exact sequence
$$
\xymatrix{
0 \ar@{->}[r] & \mi M \ar@{->}[r]^{\iota} & M \ar@{->}[r]^{f} & S(i) \ar[r] & 0
}
$$
with $\pi_N\circ\iota|_N=\id_N$. Furthermore, by assumption,  v we have $f\circ \phi =0.$
\end{proof}

In what follows we abbreviate
$$h_{S(i)}(-):=\Hom_{\mathbb{C}\Qu}(-,S(i)).$$
For $M_1,M_2 \in \mathbb{C}\Qu-\modd$, the functor $h_{S(i)}(-)$ yields the linear map
\begin{align*} h_{S(i)}: \Hom_{\mathbb{C}\Qu}(M_1,M_2) & \rightarrow \Hom\left(h_{S(i)}(M_2),h_{S(i)}(M_1)\right)\\
\phi & \mapsto h_{S(i)}(\phi).
\end{align*}

We obtain the following as a reformulation of Lemma \ref{krit}.
\begin{cor}\label{krite} Let $\phi \in \Hom_{\mathbb{C}\Qu}(\tau^{-1}M,M)$ and assume that
$$\Ker \left( h_{S(i)}(\phi)\right) \bigcap h_{S(i)}(M^{\kg})\backslash \displaystyle\bigcup_{j\in\mathscr{V}} \left(\Ker \left( h_{S(i)}(\phi)\right)\bigcap h_{S(i)}\left(V_j^{\perp}\oplus N^{-}\right)\right)\ne\varnothing.$$
Then $\phi$ descends via $V^{M}$.
\end{cor}

For $k\in \mathscr{V}$ and $\phi\in\Hom_{\mathbb{C}\Qu}(\tau^{-1}M,M)$, we denote by
$$\pi_{k}:M\twoheadrightarrow \displaystyle V_k^{\perp} \oplus N^{-}$$
the canonical projection. Recall that for $B\in\mathbb{C}\Qu-\modd$, we have $\ell_i(B)=\dim\Hom_{\mathbb{C}\Qu}(B,S(i))$. We define for $k\in \mathscr{V}$
$$\nu_k:=\min_{\phi}\ell_i(\Coker \left(\pi_k \circ \phi\right)).$$

We further define
$$\nu^{-}:=\min_{\phi} \ell_i(\Coker\left( \pi_{M^{\kg}} \circ \phi)\right).$$

\begin{lem}\label{minimum} We have
\begin{align*}
\nu^{-}&= a_i^{\Qu}(M), \text{ and }\\
\nu_k& \leq a_i^{\Qu}(M)-1
\end{align*}
for each $k\in \mathscr{V}$.
\end{lem}

\begin{proof}
Let $\phi_0\in \Hom_{\mathbb{C}\Qu}(\tau^{-1}M,M)$ be $\preceq$-minimal. By Proposition \ref{f=e} we have $\ell_i(\Coker \pi_{M^{\kg}} \circ \phi_0)=a_i^{\Qu}(M)$ implying $\nu^{-}\le a_i^{\Qu}(M).$ Let $\phi\in\Hom_{\mathbb{C}\Qu}(\tau^{-1}M,M^\kg)$ be arbitrary. Setting $\widetilde{\phi} := \phi + \phi_0|_{(\tau^{-1}M)^{\ntrianglelefteq}}$ we obtain the following commutative diagram with exact columns and rows:
\begin{equation*}
\xymatrix{0 \ar[d] & & 0 \ar[d] & 0 \ar[d] & \\
(\tau^{-1}M)^{\ntrianglelefteq}\ar@{->}[rr]^{\pi_{M^{\ntrianglelefteq}}\circ \phi_0}\ar[d] & & M^{\ntrianglelefteq } \ar@{->}[r] \ar@{->}[d] &
\Coker\pi_{M^{\ntrianglelefteq}}\circ \phi_0  \ar[d] \ar[r] & 0\\
\tau^{-1}M \ar[d]  \ar@{->}[rr]^{\widetilde{\phi}} & & M \ar@{->}[r]  \ar@{->}[d]& \Coker\tilde{\phi} \ar@{->}[d] \ar[r] & 0\\
(\tau^{-1}M)^{\kg} \ar[d]\ar@{->}[rr]^{\phi} && M^{\kg} \ar[d] \ar@{->}[r]& \Coker \phi \ar[d] \ar[r]& 0 \\
0 & & 0 & 0
}
\end{equation*}
By Proposition \ref{f=e} $\ell_i (\Coker\pi_{M^{\ntrianglelefteq}}\circ \phi_0)=0$ which implies that the map
$$h_{S(i)}( \Coker \phi) \rightarrow h_{S(i)}\left(\Coker \widetilde\phi\right)$$
is an isomorphism. Hence $\ell_i (\Coker \phi)\ge a_i^{\Qu}(M).$

By Lemma \ref{weissepunkte} there exists $\psi_0 \in \Hom_{\mathbb{C}\Qu}(\tau^{-1}M,M)$
with $$\ell_i (\Coker \pi_k \circ \psi_0)=a_i^{\Qu}(M)-1.$$
We thus obtain $\nu_k\le a_i^{\Qu}(M)-1.$
\end{proof}

We define the following dense open subset of $\Hom_{\mathbb{C}\Qu}(\tau^{-1}M,M)$:
$$\mathscr{O}_M:=\{\phi\in\Hom_{\mathbb{C}\Qu}(\tau^{-1}M,M) \mid \ell_i (\Coker\left( \pi_k \circ \phi\right))=\nu_k\}.$$

\begin{lem}\label{absteigen} For each $\phi\in\mathscr{O}_M$, we have
$$\left( \Ker \left( h_{S(i)}(\phi)\right) \cap h_{S(i)}(M^{\kg})\right)\backslash \displaystyle\bigcup_{j\in\mathscr{V}} \left(\Ker \left( h_{S(i)}(\phi)\right)\cap h_{S(i)}(V_j^{\perp}\oplus N^{-}\right)\ne\varnothing.$$
In particular, $\phi$ descends via $V^M$.
\end{lem}
\begin{proof}
The claim follows from Lemma \ref{minimum}, noting that $$\left( \Ker \left( h_{S(i)}(\phi)\right)\cap h_{S(i)}\left(M^{\kg}\right)\right)\cong h_{S(i)}(\Coker \pi_{M^{\kg}}\circ \phi)$$ is an affine variety of dimension $a_i^{\Qu}(M)$ and $$\left( \Ker\left( h_{S(i)}(\phi)\right)\cap h_{S(i)}\left(V_j^{\perp}\oplus N^{-}\right)\right)\cong h_{S(i)}(\Coker \pi_k \circ \phi)$$ is an affine variety of dimension at most $a_i^{\Qu}(M)-1$ for all $k\in \mathscr{V}$.

By Corollary \ref{krite}, we conclude that $\phi$ descends via $V^{M}$.
\end{proof}

Let $\varepsilon_i(\overline{\mathcal{C}_{[M]}})=c>0$. We say that $\phi \in \Hom_{\mathbb{C}\Qu}(\tau^{-1}M,M)$ is \emph{compatible with $\mi ^c M$} if there exists a short exact sequence
$$ 0 \rightarrow \mi^c M \xrightarrow{\iota} M \xrightarrow{f} S(i)^{c} \rightarrow 0$$
and $\psi \in \Hom_{\mathbb{C}\Qu}(\tau^{-1}\left(\mi^c M\right), \mi^{c}M)$ such that the following diagram commutes
$$
\xymatrix{ 0 \ar[r] & \mi^c M \ar[r]^{\iota} & M \ar[r]^{f} & S(i)^c \ar[r] & 0 \\
& \tau^{-1}\left(\mi^{c}M\right) \ar[u]^{\psi} \ar[r]^{\tau^{-1}\iota} & \tau^{-1} M \ar[u]^{\phi} \ar[r] & 0. \ar[u]
}
$$
To construct a dense subset of $\Hom_{\mathbb{C}\Qu}(\tau^{-1}M,M)$ which is compatible with $\m_i^{c} M$ we analyze the relationship between $X=\m_i M$ and $M$ further. We first deduce from Theorem \ref{reinekethm} that $F_i(X,V^ {M})=a_i^{\Qu}(X)$ and thus, since Lemma \ref{V0} yields that $V^{X}$ is the unique $\kg$-minimal element of $\mathscr{S}_i(\Qu)$ such that $F_i(X,V)=a_i^{\Qu}(X)$, we have
\begin{equation}\label{eq:ungleichungm}
V^{X}\kg V^{M}.
\end{equation}
Equation (\ref{eq:ungleichungm}) allows us to make the following observation.
\begin{lem}\label{directsummand} The vector space $\Hom_{\mathbb{C}\Qu}((\tau^{-1}X)^{\kg V^X},X^{\kg V^X})$ is a direct summand of the vector space $\Hom_{\mathbb{C}\Qu}(\tau^{-1}M,X)$.
\end{lem}
\begin{proof}
We write $X=\mi M = U^{M}\oplus M'$ where $M=M'\oplus V^M$ (compare with Definition \ref{operatorm}). Now (\ref{eq:ungleichungm}) shows that $\tau^{-1}U^M=\left(\bigoplus_{B\in l(V^{M})}B\right)\ntrianglelefteq V^{X}$ and $\tau^{-1}V^{M}\ntrianglelefteq V^{X}$ which proves the claim.
\end{proof}

\begin{prop}\label{dichtemenge} There exists a dense subset $\mathscr{D}_c$ of $\Hom_{\mathbb{C}\Qu}(\tau^{-1}M,M)$ such that each $\phi\in\mathscr{D}_c$ is compatible with $\mi^{c} M$.
\end{prop}
\begin{proof} Let $$\pr_M:\Hom_{\mathbb{C}\Qu}(\tau^{-1}M,M)\twoheadrightarrow \Hom_{\mathbb{C}\Qu}((\tau^{-1}M)^{\kg V^M},M^{\kg V^M})$$ be the canonical projection.

We prove the following statement by induction on $\varepsilon_i(\overline{\mathcal{C}_{[M]}})=c$:
\begin{center}
There exists a dense subset $\mathscr{D}_c$ of $\Hom_{\mathbb{C}\Qu}(\tau^{-1}M,M)$ which is compatible with $\mi M$ such that $$\pr_M^{-1}(\pr_M(\mathscr{D}_c))=\mathscr{D}_c.$$
\end{center}
If $\varepsilon_i(\overline{\mathcal{C}_{[M]}})=1$, the set $\mathscr{O}_M\subset \Hom_{\mathbb{C}\Qu}(\tau^{-1}M,M)$ has the claimed property.

Assume that $\varepsilon_i(\overline{\mathcal{C}_{[M]}})=c+1$ and let $X:=\mi M$. By induction hypothesis there exists a dense subset $\mathscr{D}_c \subset \Hom_{\mathbb{C}\Qu}(\tau^{-1}X,X)$ which is compatible with $\mi^{c}X$ and satisfies the equation
$$\pr_X^{-1}(\pr_X(\mathscr{D}_c))=\mathscr{D}_c.$$
Recall that $\Hom_{\mathbb{C}\Qu}\left((\tau^{-1}X)^{\kg V^X},X^{\kg V^X}\right)$ is a direct summand of $\Hom_{\mathbb{C}\Qu}(\tau^{-1}M,X)$ from Lemma \ref{directsummand}
and denote by $\pr_{M,X}:\Hom_{\mathbb{C}\Qu}(\tau^{-1}M,X) \twoheadrightarrow \Hom_{\mathbb{C}\Qu}\left((\tau^{-1}X)^{\kg V^X},X^{\kg V^X}\right)$ the canonical projection.

Let $\mathscr{I}$ be the set of all injective morphisms $\iota: X \hookrightarrow M$ such that $\pi_N\circ\iota|_{N}=\id_N$ and let $\Lambda(M,X)=\Hom_{\mathbb{C}\Qu}(\tau^{-1}M,X)\times \mathscr{I}$ be the variety of pairs $(\xi,\iota)$ where $\xi \in \Hom_{\mathbb{C}\Qu}(\tau^{-1}M,X)$ and $\iota \in \mathscr{I}$. We define \begin{align*} p_1:\Lambda(M,X) & \rightarrow \Hom_{\mathbb{C}\Qu}(\tau^{-1}M,X) \\
(\xi,\iota)&\mapsto \xi
\end{align*}
and
\begin{align*}
p_2:\Lambda(M,X) & \rightarrow \Hom_{\mathbb{C}\Qu}(\tau^{-1}M,M) \\
(\xi,\iota)&\mapsto \iota\circ \xi.
\end{align*}
This yields the following diagram
\begin{center}\resizebox{14,5cm}{0.33cm}{$
\xymatrix{& & & \resizebox{3cm}{!}{$\Lambda(M,X)$}\ar@{->}[ld]_{{\resizebox{0.5cm}{!}{$p_1$}}}\ar@{->}[rd]^{{\resizebox{0.5cm}{!}{$p_2$}}}  \\ \resizebox{5,9cm}{!}{$\Hom_{\mathbb{C}\Qu}(\tau^{-1}X,X)$}\ar@{->}[dr]^{{\resizebox{0.8cm}{!}{$\pr_X$}}} & & \resizebox{5,9cm}{!}{$\Hom_{\mathbb{C}\Qu}(\tau^{-1}M,X)$} \ar@{->}[dl]_{{\resizebox{1,2cm}{!}{$\pr_{M,X}$}}} & & \resizebox{5,9cm}{!}{$\Hom_{\mathbb{C}\Qu}(\tau^{-1}M,M).$}  \\
 & \resizebox{9,0cm}{!}{$\Hom_{\mathbb{C}\Qu}\left((\tau^{-1}X)^{\kg V^X},X^{\kg V^X}\right)$} & &
}
$}
\end{center}
We define $$\mathscr{D}_{c+1}:=p_2\left(p_1^{-1}\left(\pr_{M,X}^{-1}\left(\pr_X\left(\mathscr{D}_{c}\right)\right)\right)\right).$$
Note that $\pr_X$, $pr_{M,X}$ and $p_1$ are projections and therefore continuous, surjective and open. Further, by Lemma \ref{absteigen}, $\mathscr{O}_M\subset \im p_2$ and thus $p_2$ is continuous with dense image. Hence $\mathscr{D}_{c+1}$ is a dense subset of $\Hom_{\mathbb{C}\Qu}(\tau^{-1}X,X)$.

We show that $\pi_M^{-1}\pi_M(\mathscr{D}_{c+1})=\mathscr{D}_{c+1}$. Let $\widetilde{\phi}\in \pi_M^{-1}\pi_M(\mathscr{D}_{c+1})$. By construction there exists $\phi \in \mathscr{D}_{c+1}$ and $\lambda \in \Ker\pi_{M}$ such that $\widetilde{\phi}=\phi + \lambda $. Since by (\ref{eq:ungleichungm}) we have $p_1 (p_2^{-1} ( \Ker \pi_M)) \subseteq \Ker \pr_{M,X}$, it follows that $\widetilde{\phi}\in\mathscr{D}_{c+1}$.

It remains to show that any $\phi\in \mathscr{D}_{c+1}$ is compatible with $\mi^{c+1}$. Let therefore $\bar{\phi}\in \pr_{M,X}^{-1}\pr_X(\mathscr{D}_c)$ and $\iota \in \mathscr{I}$ such that
$$\phi=\iota\circ \bar{\phi}.$$
We show that $\bar{\phi}\circ \tau^{-1}\iota \in \mathscr{D}_c$ which is equivalent to
\begin{equation}\label{eq:mx}
\pr_X(\bar{\phi}\circ \tau^{-1}\iota)=\pr_{M,X}(\bar{\phi}).
\end{equation}
The claim then follows from the induction hypothesis.

Equation (\ref{eq:mx}) holds since by (\ref{eq:ungleichungm}) $$\Hom_{\mathbb{C}\Qu}(\tau^{-1}V^X,N^{\kg V^X})=0=\Hom_{\mathbb{C}\Qu}(\tau^{-1}V^M,N^{\kg V^X})$$ and $\pi_{\tau^{-1}N}\circ\tau^{-1}\iota|_{\tau^{-1}N}=\id_{\tau^{-1}N}$ by the defition of $\mathscr{I}$.
\end{proof}
\begin{prop}\label{crystaliso} For $M\in \mathbb{C}\Qu-\modd$, with $\varepsilon_i(\overline{\mathcal{C}_{[M]}})=c>0$, we have
$$\tilde{e}_i^c \overline{\mathcal{C}_{[M]}}=\overline{\mathcal{C}_{[\mi^c M]}}.$$
\end{prop}

\begin{proof}

We have shown in Proposition \ref{dichtemenge} that there exists a dense subset $\mathscr{D}_c\subset \Hom_{\mathbb{C}\Qu}(\tau^{-1}M,M)$  such that for every $\phi\in \mathscr{D}_c$,  the (up to isomorphism) unique $\Pi(\Qu)$-submodule of $(M,\phi)$ with quotient isomorphic to $S(i)^c$ is of the form $(\mi^c M, \psi)$ for a $\psi \in \Hom_{\mathbb{C}\Qu}(\tau^{-1} \mi^c M, \mi^c M)$. Recall from Theorem \ref{preprojmodule} that $\Hom_{\mathbb{C}\Qu}(\tau^{-1}M,M)$ can be identified with the fiber of $M$ of the conormal bundle $\mathcal{C}_{[M]}$. Thus $\mathscr{D}_c$ can be identified with a dense subset of that fiber. Now $g\in G_v$ maps $(M,\phi)$ to an isomorphic $\Pi(\Qu)$-module and the claim thus follows from Remark \ref{dense} and Remark \ref{fiber}.
\end{proof}
Since the weight map is clearly preserved by $\mathscr{F}$ we have thus proved:

\begin{thm}\label{mainresult} For $\Qu$ a special Dynkin quiver the map $\mathscr{F}:B^{\mathscr{H}}(\infty)\rightarrow B^g(\infty)$ given by ($M\in \Repp_V(\Qu)$) $$b_{[M]}\mapsto \overline{\mathcal{C}_{[M]}}$$ is an isomorphism of crystals.
\end{thm}

\end{document}